\documentclass{amsart}
\usepackage{amsmath,verbatim,amsthm,amssymb,accents}
\usepackage{graphicx,tensor}
\usepackage[all]{xy}
\usepackage{tikz-cd}
\allowdisplaybreaks[4]
\theoremstyle{theorem}
\usepackage[latin1]{inputenc}
\allowdisplaybreaks[4]
\newtheorem{theorem}{Theorem}[section]
\newtheorem{proposition}[theorem]{Proposition}

\newtheorem{lemma}[theorem]{Lemma} 
\theoremstyle{definition} 
\newtheorem{remark}[theorem]{Remark} 

\numberwithin{equation}{section} 
\newcommand\x[1]{\mathrm{#1}}
\renewcommand\o[1]{\overline{#1}}

\newcommand\h[1]{\hat{#1}}

\renewcommand\r[1]{\mathring{#1}}

\renewcommand\t[1]{\tilde{#1}}

\newcommand\f[2]{\frac{#1}{#2}}
\newcommand\I{\int\limits}
\renewcommand\({\Big(}
\renewcommand\){\Big)}
\renewcommand\S{\sum\limits}
\renewcommand\P{\prod\limits}
\newcommand\U{\bigcup\limits}

\newcommand\F[1]{\sqrt{#1}}

\newcommand\AL{\mathcal A}
\newcommand\BL{\mathcal B}
\newcommand\CL{\mathcal C}

\newcommand\FL{\mathcal F}

\newcommand\HL{\mathcal H}

\newcommand\LL{\mathcal L}

\newcommand\PL{\mathcal P}

\newcommand\TL{\mathcal T}

\newcommand\Bl{\mathbf B}
\newcommand\Cl{\mathbf C}

\newcommand\Il{\mathbf I}
\newcommand\Kl{\mathbf K}
\newcommand\Ml{\mathbf M}
\newcommand\Nl{\mathbf N}

\newcommand\Rl{\mathbf R}
\newcommand\Sl{\mathbf S}
\newcommand\Tl{\mathbf T}

\newcommand\la{\alpha}
\newcommand\lb{\beta}

\renewcommand\lg{\gamma}
\newcommand\lG{\Gamma}
\newcommand\ld{\delta}
\newcommand\lD{\Delta}
\newcommand\Le{\varepsilon}

\newcommand\lk{\kappa}
\newcommand\lf{\phi}

\newcommand\lF{\Phi}
\renewcommand\ll{\lambda}
\newcommand\lL{\Lambda}
\newcommand\lm{\mu}
\renewcommand\ln{\nu}

\newcommand\lO{\Omega}
\newcommand\lp{\pi}

\renewcommand\lq{\psi}

\newcommand\ls{\sigma}

\newcommand\lt{\vartheta}

\newcommand\lz{\zeta}
\newcommand\m{{\boldsymbol m}}
\newcommand\n{{\mathbf n}}

\renewcommand\.{\cdot}
\renewcommand\:{\cdots}
\renewcommand\;{\ldots}
\newcommand\dl{\partial}
\newcommand\sm{\setminus}
\newcommand\oo{\infty}
\newcommand\oc{\circ}

\newcommand\op{\oplus}
\newcommand\xx{\times}

\newcommand\ic{\subset}
\newcommand\ci{\supset}
\newcommand\iu{\cup}
\newcommand\ui{\cap}

\renewcommand\l{\ell}
\newcommand\nl{\natural}
\newcommand\er{\eqref} 
\newcommand\be[1]{\begin{equation}\label{#1}}
\newcommand\ee{\end{equation}}
\newcommand\bi{\begin{itemize}}
\newcommand\ei{\end{itemize}}

\newcommand\z[3]{\tensor[_{#1}]{#2}{_{#3}}}

\begin{document}

\title[Boundary orbits]{Toeplitz $C^*$-Algebras on Boundary Orbits of Symmetric Domains}
\author[G. Misra, H. Upmeier]{Gadadhar Misra {\rm and} Harald Upmeier} 
\medskip

\address{Department of Mathematics, Indian Institute of Science, Bangalore 560012, India}
\email{gm@iisc.ac.in}

\address{Fachbereich Mathematik, Universit\"at Marburg, D-35032 Marburg, Germany}
\email{upmeier{@}mathematik.uni-marburg.de}

\dedicatory{This paper is dedicated to Nikolai Vasilevski on the occasion of his 70th birthday}

\thanks{The first-named author was supported by the J C Bose National Fellowship of the DST and CAS II of the UGC and the second-named author was supported by an Infosys Visiting Chair Professorship at the Indian Institute of Science}
\medskip

\subjclass{Primary 32M15, 46E22; Secondary 14M12, 17C36, 47B35}

\keywords{symmetric domain, boundary orbit, algebraic variety, weighted Bergman spaces, Toeplitz operator, subnormal operator, $C^*$-algebra}

\begin{abstract} We study Toeplitz operators on Hilbert spaces of holomorphic functions on symmetric domains, and more generally on certain algebraic subvarieties, determined by integration over boundary orbits of the underlying domain. The main result classifies the irreducible representations of the Toeplitz $C^*$-algebra generated by Toeplitz operators with continuous symbol. This relies on the limit behavior of ``hypergeometric" measures under certain peaking functions.
\end{abstract}

\maketitle

\setcounter{section}{-1}

\section{\bf Introduction}
Toeplitz operators and Toeplitz $C^*$-algebras on Hilbert spaces over bounded symmetric domains $\lO=G/K,$ for a semisimple Lie group 
$G$ and a maximal compact subgroup $K,$ are a deep and interesting part of multi-variable operator theory \cite{U2,U3,U4}, closely related to harmonic analysis (holomorphic discrete series of representations of $G$) and index theory. In this paper we study Hilbert spaces over non-symmetric $G$-orbits contained in the {\bf boundary} of $\lO.$ These Hilbert spaces do not belong to the holomorphic discrete series, but the associated Toeplitz operators are still $G$-homogeneous in the sense of \cite{M}. We study the $C^*$-algebra generated by these Toeplitz operators on boundary orbits and construct its irreducible representations, similar as in the symmetric case, via a refined analysis of the boundary faces of these orbits. The most interesting discovery is that for the boundary Toeplitz $C^*$-algebra, the irreducible representations do not always belong to boundary orbits, but comprise also some distinguished parameters in the discrete series (relative to the face). 

Recently, certain algebraic varieties in symmetric domains, called {\bf Jordan-Kepler varieties}, have been studied from various points of view \cite{EU,U4}. Although these varieties are not homogeneous, there exist natural $K$-invariant measures giving rise to Hilbert spaces of holomorphic functions and associated Toeplitz operators. In \cite{U5} the corresponding Toeplitz $C^*$-algebra and its representations have been investigated using asymptotic properties of hypergeometric functions. As a second main result of this paper, we combine both settings and treat Kepler-type varieties related to boundary orbits. The associated Toeplitz operators are subnormal, but the explicit description of the underlying boundary measure requires some effort. It seems that our setting is the natural level of generality, where methods of harmonic analysis based on Jordan algebraic concepts still yield a complete structure theory of Toeplitz $C^*$-algebras.

Compared to the paper \cite{U5}, to which we frequently refer, the main new result concerns the description of the measures and inner product for the underlying Hilbert space, and the expression of the reproducing kernel in terms of generalized  hypergeometric series. For boundary orbits  this is not straightforward. Also, the concept of ``hypergeometric measure" introduced in Section \ref{s} serves to clarify and streamline the exposition, especially in the proof of Theorem \ref{t}.

\section{\bf Subnormal and homogeneous operator tuples}
To put the results of this paper in perspective, recall that a commuting $n$-tuple of operators $\Sl =(S_1,\;,S_n)$ is said to be {\bf subnormal} if it is the restriction of a commuting tuple of normal operators $\Nl,$ acting on a Hilbert space $\HL,$ to an invariant subspace $\HL_0\ic\HL$. There are several intrinsic characterizations of subnormality; the one closest to the spirit of this paper is the following {\bf $C^*$-algebraic characterization}. Let $C^*[\Sl]$ be the $C^*$-algebra generated by $\{\text{Id},\,S_1, \ldots , S_n\}$
\begin{theorem}[{\cite[Theorem 2]{L}}] A commuting $n$-tuple of operators $\Sl$ is subnormal if and only if for every  subset 
$\{T_I:I\in \FL\}$ of $C^*[\Sl]$, $\FL$ finite, it follows that 
$$\S_{I,J\in \FL}T_I^*{\Sl^J}^* \Sl^I T_J\ge 0,$$
where $T_I=T_{i_1}\:T_{i_n}$ and $\Sl^I=S_1^{i_1}\:S_n^{i_n}.$  
\end{theorem}
An immediate corollary is that if $\Sl$ is a subnormal commuting $n$-tuple and $\lp$ is a $*$-representation of the $C^*$-algebra 
$C^*[S]$, then $\lp(\Sl)$ is also subnormal. For $n=1$, these results were obtained by Bunce and Deddens \cite{BD}. 
Natural examples of subnormal operators are obtained by restricting the multiplication by the coordinate functions on the Hilbert space $L^2(\lO,m)$ to the subspace of holomorphic functions $H^2(\lO,m)$, where $\lO\ic\Cl^d$ is a bounded domain and $m$ is a finite measure supported in the closure $\o\lO$ of $\lO$. Determining when a commuting tuple of operators is subnormal, in general, is not easy. For instance, let $\lO$ be a bounded symmetric domain of genus $p$, and let $B$ be the Bergman kernel of $\lO$. Then the set of positive real $\ln$ for which $B^{\ln/p}$ remains a positive definite kernel is known (cf. \cite{FK1}) and is designated the {\bf Wallach set} of $\lO$.  For a fixed but arbitrary $\ln$ in the Wallach set, let $\HL^{(\ln)}$ denote the  Hilbert space determined by $B^{\ln/p}$. 
The biholomorphic functions of the domain $\lO$ form a group, say $G$. Thus $g\in G$ acts on $\lO$ via the map $(g,z)\mapsto g(z)$. This action lifts ($g \mapsto U_g,\, g\in G$) to the Hilbert space $\HL^\ln$: 
$$\(U^{(\ln)}_{g^{-1}}f\)(z)=Jg(z)^{\ln/p}\(f(g(z)\), \,\ g\in G,\ z\in\lO,\ f\in\HL^{(\ln)},$$    
where $Jg(z):=\det (D g(z))$. It is easy to verify, using the transformation rule for the Bergman kernel, that $U_g$ is unitary. The map $g\to U^{(\ln)}_g$ is not a homomorphism, in general, however $U^{(\ln)}_{gh}=c(g,h) U^{(\ln)}_g U^{(\ln)}_h$, where 
$c:G\xx G\to\Tl$ is a Borel multiplier. Thus $U$ defines a projective unitary representation of the group on $\HL^{(\ln)}$.  

The automorphism group $G$ admits the structure of a Lie group. Consider the bounded symmetric domain $\lO$ in its Harish-Chandra realization (cf. \cite[Section 2.1]{I}). The construction of the {\bf discrete series representations} due to Harish-Chandra is well known, see \cite[Theorem 6.6]{Kn}. The (scalar holomorphic) discrete series representations (when realized as sections of homogeneous holomorphic line bundles) occur among the projective unitary representations $U^{(\ln)}$. Harish-Chandra had determined a cut-off 
$\ln_1$ such that for all $\ln>\ln_1$, the representation $U^{(\ln)}$ is in the discrete series and the Hilbert space $\HL^{(\ln)}$ is realized as the space $H^2(\lO,dm_\ln)$, where $dm_\ln(z)=B(z,z)^{1-\ln/p}\ dv(z)$, clearly, $\x{supp}(m)=\o\lO$.  However, we also have the so-called limit discrete series representations and their analytic continuation. It is therefore natural to ask if there are other values of $\ln$ for which the inner product in the Hilbert space $\HL^{(\ln)}$ is given by an integral with respect to a measure supported on possibly some other $G$-invariant closed subset of $\o\lO$. The answer to this question involves the $G$-invariant {\bf boundary strata} of $\o\lO$ introduced below, namely, $\o\lO_{k,r}$, $1\le k\le r$, where  $r$ is the rank of the bounded symmetric domain $\lO$. 
In this notation, $\lO_{r,r}$ is the Shilov boundary and $\lO_{0,r}=\lO$. For $\ln$ in $\{\ln_1,\;,\ln_r\}$, where 
$$\ln_i=\tfrac{d}{r}+\tfrac{a}{2}(r-i),$$ 
there exists a quasi-invariant measure 
$$dm_i(gz)=|J g(z)|^{\tfrac{2\nu_i}{p}}dm_i(z),\,\ z\in\lO,\ \x{supp}(m_i)=\o\lO_{i,r},\ 1\le i\le r,$$ 
such that $L^2(\lO_{i,r},dm_i)$ contains the representation space $\HL^{(\ln)}$ as a closed subspace. (Here, with a slight abuse of notation, we let $\lO_{0,r}=\o\lO$.) The representation $U^{(\ln)}$ lifts to $\widehat{U^{(\ln)}}$ on $L^2(\lO_{i,r},dm_i)$, again, as a multiplier representation, see \cite[theorem 6.1]{AZ}. The existence of the quasi-invariant measure (in the unbounded realization of 
$G/K$) is in \cite{RV,La2}, see also \cite[Lemma 5.1]{BM1}. (The generalization to the case of vector valued holomorphic functions appears in \cite[Theorem 4.49]{I}.) However, the fact that these are the only quasi-invariant measures with support in $\o\lO$ was proved for the domains $\lO$ of type $I_{n,m}$, $m\ge n\ge 1,$ in \cite{BM1} and was extended to all bounded symmetric domains in 
\cite{AZ}. Furthermore, it can be shown that these are the only commuting tuples of ``homogeneous'' subnormal  operators in the Cowen-Douglas class of rank $1$ on $\lO$. 

Thus the commuting tuple $\Ml^{(\ln)}:=(M_1^{(\ln)},\;,M_d^{(\ln)})$ of multiplication by the coordinate functions on the Hilbert space $\HL^{(\ln)}$ is subnormal if and only if $\ln$ is in the set 
$$W_{\rm sub}:=\{\ln:\ \ln=\tfrac{d}{r}+\tfrac{a}{2}(r-j),\ 1\le j\le r\}\iu\{\ln:\ln>p-1\}.$$
For $\ln$ as above, this is evident since the Hilbert space $\HL^{(\ln)}$ is a closed subspace of the Hilbert space $L^2(dm_\ln)$ for some quasi-invariant measure $m_\ln$. The converse is Theorem 3.1 of \cite{BM1} for tube type domains and Theorem 5.1 of \cite{AZ} in general.

The commuting tuple $\widehat\Ml$ of multiplication  by the coordinate functions on the Hilbert space $L^2(dm_\ln)$ induces a $*$ - homomorphism $\widehat\lF_\ln:\CL(\lO_{i,r})\to\LL(L^2(dm_\ln))$, namely, $\widehat\lF_\ln(f)=f(\widehat\Ml)$, $f\in\CL(\lO_{i,r})$, the space of continuous functions on $\lO_{i,r}$ and $\ln\in W_{\rm sub}$. The quasi-invariance of the measure $m_\ln$ ensures that 
$\widehat U^{(\ln)}$ is unitary and therefore the triple $(L^2(dm_\ln),\widehat U^{(\ln)},\widehat\lF_\ln)$ is a {\bf system of imprimitivity} in the sense of Mackey \cite[chapter 6]{V}:  
\be{imp}
(\widehat U_g^{(\ln)})^*\widehat\lF_\ln\ \widehat U^{(\ln)}=g\.\widehat\lF_\ln,\,\ g\in G,\ee
where $((g\.\widehat\lF_\ln)f)(z)=f(g\.z).$ Since the representation $\widehat U^{(\nu)}$ leaves the subspace $\HL^{(\nu)}$ invariant as well, we see that 
$$(\HL^{(\ln)},U^{(\ln)},\lF_\ln)=(L^2(dm_\ln),\widehat U_g^{(\ln)},\widehat\lF_\ln)_{|\,\HL^{(\ln)}}, \,\,\ln\in W_{\rm sub},$$ 
is the restriction of an imprimitivity. 

Recall that the $*$-homomorphism $\widehat\lF$ must be given by the formula $\widehat\lF(f)=\widehat\Ml_f=f(\widehat\Ml)$, 
$f\in\CL(\lO_{i,r})$, $0\le i\le r$, via the usual functional calculus. The group $G$ acts on the space of continuous functions via 
$(g^{-1}\.f)(z)=f(g\.z)=(f\oc g)(z)$. Therefore, 
$$\widehat\lF(g\.f)=\widehat\Ml_{f\oc g}=(f\oc g)(\widehat\Ml).$$ 
Choosing $f$ to be the coordinate functions, we see that the imprimitivity condition \er{imp} of Mackey is equivalent to the homogeneity of the commuting tuple $\Ml$, relative to the group $G$, of the commuting tuple $\widehat\Ml$, namely,
\be{homog}
U_g\Ml U_g^*:=(U_gM_1U_g^*,\;,U_g^* M_d U_g)=g\.\Ml,\,\, g\in G,\ee
where $g\.\Ml=(g_1(\Ml),\;,g_d(\Ml))$. Here $g_i$, $1\le i\leq d$, are the components of $g$ in $G$, when it is thought of as an injective biholomorphic map on $\lO$.  This notion for a single operator is from \cite{M} and for a commuting tuple is from \cite{MS}, see also \cite{BM1,BM2}. 
For $\ln$ in the Wallach set, the multiplication by the coordinate functions acting on the Hilbert space of holomorphic functions 
$\HL^{(\ln)}$ are bounded if and only if $\ln\in(\tfrac{a}{2}(r-1),\oo)$, the continuous part of the Wallach set, see \cite[Theorem 4.1]{AZ} and \cite[Theorem 1.1]{BM1}. Since the kernel function of the Hilbert space $\HL^{(\ln)}$ is a power of the Bergman kernel, it also transforms like the Bergman kernel ensuring that the the operator $\Ml$ on this Hilbert space is $G$-homogeneous for all $\ln$ in the continuous part of the Wallach set. A simple computation involving the curvature shows that these are the only $G$-homogeneous operators in the Cowen-Douglas class $B_1(\lO)$. The details are in \cite{MS} for the case of rank $r=1$. The proofs in the general case can be obtained using \cite[Proposition 4.4]{AZ} and spectral mapping properties of the Taylor spectrum of the commuting tuple 
$\Ml$. 

It is clearly of interest to study homogeneity, or equivalently, imprimitivity relative to {\bf subgroups} of the group $G.$ This already occurs in the study of spherically balanced tuples of operators \cite[Definition 1.1]{CY}. In this case, the domain is the Euclidean unit ball $\Bl_d$ and the group is the maximal compact subgroup $K$ of the automorphism group $G$ of $\Bl_d$. The group $K$ can be identified with the unitary group $U(d)$, it acts on $\Bl_d$ by the rule: 
$(U,z)\mapsto U(z)$, $z\in\Bl_d,\ U\in U(d)$. Let $\Tl$ be a commuting $d$-tuple of operators acting on a complex separable Hilbert space $\HL$. The usual functional calculus gives  
$$U\.\Tl=\(\S_{j=1}^d U_{1j}T_j,\;,\S_{j=1}^d U_{d,j}T_j\),\,\, U\in K.$$  
The commuting $d$-tuple of operators $\Tl$ is said to be ``spherically symmetric'', or equivalently, $K${\bf-homogeneous} if 
$\lG_U^*\Tl\lG_U=U\.\Tl$ for each $U$ in $K$ and some unitary $\lG_U$ on $\HL$. In general, $\lG$ need not be a unitary representation. However, we will assume that a choice of $\lG_U$ exists such that the map $U\to\lG_U$ is a unitary homomorphism. What we have said about the Euclidean ball applies equally well to the case of a bounded symmetric domain. So, we speak freely of $K$-homogeneous operators, where $\lO=G/K$. To describe this more general situation, we recall some basic notions from the representation theory of the group $K$. 

Let $\m\in\Nl_+^r$ be a partition of length $r$. Let $\PL_\m$ denote the space of irreducible $K$-invariant homogeneous polynomials of isotypic type $\m,$ having total degree $|\m|$. These are mutually inequivalent as $K$-modules and $\PL=\S_{\m\in\Nl_+^r}\PL_\m$ is the Peter-Weyl decomposition of the polynomials $\PL$ under the action of the group $K$. Now, equip the submodules $\PL_\m$ with the Fischer-Fock inner product $(p|q)_\m=(q^*(\dl)(p))(0)$, where $q^*(z)=\o{q(\o z)}.$ Let $E^\m$ be the reproducing kernel of the finite dimensional space $\PL_\m$. Then the Faraut-Kor\'anyi formula for the reproducing kernel $K^{(\ln)}$ of the Hilbert space $\HL^{(\ln)}$ is 
\be{7}K^{(\ln)}=\S_{\m\in\Nl_+^r}(\ln)_\m E^\m,\ee
where $(\ln)_\m:=\P_{j=1}^r(\ln-\f a2(j-1))_{m_j}$ are the generalized Pochhammer symbols. We have pointed out that the commuting tuple of multiplication operators $\Ml$ on the Hilbert space $\HL^{(\ln)}$ is $G$-homogeneous, therefore, it is also $K$-homogeneous. What are the other $K$-homogeneous operators? Since $\PL_\m$ is a $K$ irreducible module, it follows that the Hilbert space 
$\HL^{(\boldsymbol a)},$ obtained by setting $K^{(\boldsymbol a)}=\S_{\m\in\Nl_+^r}a_\m E^\m$ for an arbitrary choice of positive numbers $a_\m$ is a weighted direct sum of the $K$ modules $\PL_\m$.  Hence the commuting tuple of multiplication operators $\Ml$ on 
$\HL^{(\boldsymbol a)}$ is $K$-homogeneous. It is shown in \cite{GKP}, under some additional hypothesis, that these are the only $K$ - homogeneous operators. 

If the rank $r=1$, then a full description of all multi-shifts within the class of spherically symmetric operators is given in 
\cite[Theorem 2.5]{CY}. In the present set-up, this characterization amounts to saying that a multi-shift on a Hilbert space $\HL$ with reproducing kernel $K:\Bl_d\xx\Bl_d\to\Cl$ is spherically symmetric if and only if the kernel is of the form 
$$\S_n a_n \langle\boldsymbol z,\boldsymbol w\rangle^n$$
for $\boldsymbol z, \boldsymbol w\in\Bl_d$. It then follows that several properties of the commuting tuple of multiplication operators 
$\Ml$ on the Hilbert space are determined by the ordinary shift with weight sequence $\big \{\(\tfrac{a_n}{a_{n+1}}\)^{1/2}\big\}$, $n\ge 0$, see \cite[Theorem 5.1]{CY}. 

\section{\bf Spectral varieties and boundary orbits}\label{r}
In this section we describe the Jordan theoretic background needed for the rest of the paper. For details, cf. \cite{FK2,Lo2}. Let $V$ be an irreducible hermitian Jordan triple of rank $r.$ Every element $z\in V$ has a {\bf spectral decomposition}
$$z=\S_{i=1}^r\ll_i c_i$$
where the singular values $\ll_1\ge\ll_2\ge\;\ge\ll_r\ge 0$ are uniquely determined by $z,$ and $c_1,\;,c_r$ is a frame of minimal orthogonal tripotents. The largest singular value $\|z\|:=\ll_1$ defines a (spectral) norm on $V$ and the (open) unit ball
$$\lO=\{z\in V:\ \|z\|<1\}$$ 
is a bounded symmetric domain. It is a fundamental fact \cite{Lo2} that, conversely, every hermitian bounded symmetric domain can be realized, in an essentially unique way, as the spectral unit ball of a hermitian Jordan triple. In this paper we use the Jordan algebraic approach to study analysis on symmetric domains and related geometric structures. 

The compact group $K$ acts transitively on the set of frames. Hence, for fixed $\ll=(\ll_1,\ll_2,\;,\ll_r),$ the {\bf level set}
\be{1}V(\ll):=\{z=\S_{i=1}^r\ll_i c_i:\ (c_i)\mbox{ frame}\}\ee
is a compact $K$-orbit. As a special case we obtain the compact manifold 
$$S_k:=V(1^k,0^{r-k})$$
of all {\bf tripotents} of rank $k,$ where $0\le k\le r.$ Every union of such level sets \er{1} is $K$-invariant but may be an orbit of a larger group. As an example, for $0\le\l\le r,$ the {\bf Jordan-Kepler manifold}  
$$\r V_\l=\U_{\ll_1\ge\;\ge\ll_\l>0}V(\ll_1,\;,\ll_\l,0^{r-\l}),$$
consisting of all elements of rank $\l,$ is a complex manifold which is an orbit under the complexified group $K^\Cl.$ Its closure
$$V_\l=\U_{\ll_1\ge\:\ge\ll_\l\ge 0}V(\ll_1,\;,\ll_\l,0^{r-\l})=\U_{0\le j\le\l}\r V_j$$
consists of all elements of rank $\le\l$ and is called the {\bf Jordan-Kepler variety}. Its regular (smooth) part coincides with 
$\r V_\l.$ For $\l=r$ we have $V_r=V$ and $\r V_r=\r V$ is an open dense subset, consisting of all elements of maximal rank. As another example the set
$$\lO_{k,r}=\U_{1>\ll_{k+1}\ge\;\ge\ll_r\ge 0}V(1^k,\ll_{k+1},\;,\ll_r)$$
is an orbit under the identity component $G$ of the biholomorphic automorphism group of $\lO.$ For $k=0,$ we have $\lO_{0,r}=\lO.$ For $k>0$ we obtain a {\bf boundary orbit} which is not a complex submanifold. It has the closure
$$\o\lO_{k,r}=\U_{1\ge\ll_{k+1}\ge\;\ge\ll_r\ge0}V(1^k,\ll_{k+1},\;,\ll_r)=\U_{i=k}^r\lO_{i,r}.$$
The intersection
$$S_k=\r V_k\ui\lO_{k,r}$$
is the common {\bf center} of $\r V_k$ and $\lO_{k,r}.$ In particular, $S_0=\{0\}$ is the center of $\lO.$ The triple
$$\r V_k\ci S_k\ic\lO_{k,r}$$
is a special case of {\bf Matsuki duality}, which gives a 1-1 correspondence between $G$-orbits and $K^\Cl$-orbits in a flag manifold 
(which in our case is the so-called conformal hull of $V$), determined by the condition that the intersection is a $K$-orbit. 
For $k=r$ we obtain the {\bf Shilov boundary}
$$\lO_{r,r}=S_r=:S$$
which is the only closed stratum of $\dl\lO$ and is its own center. Generalizing both the Jordan-Kepler varieties and the boundary orbits, we define for $0\le k\le\l\le r$ the $K$-invariant set
$$\r\lO_{k,\l}:=\r V_\l\ui\lO_{k,r}=\U_{1>\ll_{k+1}\ge\;\ge\ll_\l>0}V(1^k,\ll_{k+1},\;,\ll_\l,0^{r-\l}).$$
It has the closure
$$\o\lO_{k,\l}=V_\l\ui\o\lO_{k,r}=\U_{1\ge\ll_{k+1}\ge\;\ge\ll_\l\ge 0}V(1^k,\ll_{k+1},\;,\ll_\l,0^{r-\l})
=\U_{k\le i\le j\le\l}\r\lO_{i,j}.$$
We also use the 'partial closure' 
$$\lO_{k,\l}:=V_\l\ui\lO_k=\U_{1>\ll_{k+1}\ge\;\ge\ll_\l\ge0}V(1^k,\ll_{k+1},\;,\ll_\l,0^{r-\l})=\U_{j=k}^\l\r\lO_{k,j}.$$
Then
$$\lO_\l:=\lO_{0,\l}=V_\l\ui\lO$$
is the so-called {\bf Kepler ball}.

Our first goal is to describe a {\bf facial decomposition} of the $K$-invariant sets $\lO_{k,\l}.$ For a tripotent $c$ we consider the 
{\bf Peirce decomposition} \cite{Lo1,Lo2}
$$V=V_2^c\op V_1^c\op V_0^c.$$
Define $V^c:=V_0^c$ and $\lO^c:=\lO\ui V^c.$ This is itself a bounded symmetric domain of rank $r-k,$ when $c\in S_k.$
 
\begin{proposition}\label{d} There exist fibrations (disjoint union)
\be{9}\r\lO_{k,\l}=\U_{c\in S_k}c+\r\lO_{\l-k}^c\ic\lO_{k,\l}=\U_{c\in S_k}c+\lO_{\l-k}^c=\U_{i=k}^\l\r\lO_{k,i}\ee
\end{proposition}
\begin{proof} If $z\in\r\lO_{k,\l}$ then
$$z=c_1+\ldots+c_k+\S_{k<i\le\l}\ll_i c_i$$
for some frame $(c_i)$ and $1>\ll_{k+1}\ge\;\ge\ll_\l>0.$ It follows that $c:=c_1+\ldots+c_k\in S_k$ and 
$$w:=\S_{k<i\le\l}\ll_i c_i\in\lO^c\ui\r V_{\l-k}=\r\lO_{\l-k}^c.$$
For different tripotents $c,c'\in S_k$ the boundary components $c+\lO^c$ and $c'+\lO^{c'}$ are disjoint \cite[Section 6]{Lo2}. This proves the first assertion. If $z\in\lO_{k,\l}$ then we require only $\ll_\l\ge0.$ Therefore 
$$w\in\lO^c\ui V_{\l-k}=\lO_{\l-k}^c=\U_{i=k}^\l\r\lO_{i-k}^c.$$
It follows that
$$\lO_{k,\l}=\U_{c\in S_k}c+\lO_{\l-k}^c=\U_{c\in S_k}\U_{i=k}^\l c+\r\lO_{i-k}^c=\U_{i=k}^\l\U_{c\in S_k}c+\r\lO_{i-k}^c
=\U_{i=k}^\l\r\lO_{k,i}.$$
\end{proof}

For $k\le i\le\l$ the set $\r\lO_{k,i}$ is called the {\bf $i$-th stratum} of $\lO_{k,\l}.$ In the special case $\l=r$ we obtain a stratification
$$\lO_{k,r}=\U_{c\in S_k}c+\lO_{r-k}^c=\U_{i=k}^r\r\lO_{k,i}$$
of the boundary $G$-orbit.

\section{\bf Hypergeometric measures}\label{s}
If $V$ is an irreducible hermitian Jordan triple of rank $r$, with automorphism group $K,$ define the $K$-average
$$f^\nl(t):=\I_K dk\ f(kt)$$
for $t\in\Rl_{++}^r:=\{t\in\Rl^r:\ t_1\ge\;\ge t_r\ge0\}.$ Any $K$-invariant measure $\lm$ on $V$ (or a $K$-invariant subset) has a 
{\bf polar decomposition}
$$\I\lm(dz)\ f(z)=\I\t\lm(dt_1,\;,dt_r)\ f^\nl(\F t_1,\;,\F t_r)$$
for a uniquely defined measure $\t\lm$ on $\Rl_{++}^r$ (or a suitable subset), called the {\bf radial part} of $\lm.$ In the following we use various unspecified constants, all of which are explicitly known.

\begin{proposition} The Lebesgue measure $dz=:\ll_r(dz),$ for the normalized $K$-invariant inner product on $V,$ has the radial part
\be{4}\t\ll_r(dt_1,\;,dt_r)=\x{const.}\ \P_{i=1}^r dt_i\ t_i^b\ \P_{1\le i<j\le r}(t_i-t_j)^a\ee
on $\Rl_{++}^r.$ Here $a,b$ denote the so-called characteristic multiplicities of $V$ \cite[Section 17]{Lo1}.
\end{proposition}
\begin{proof} We start with the well known formula
\be{2}\I_X dx\ f(x)=\x{const.}\ \I_{\Rl_+^r}dt_1\:dt_r\ \P_{1\le i<j\le r}(t_i-t_j)^a\I_L dh\ f(ht)\ee
for a euclidean Jordan algebra $X$ with automorphism group $L$ \cite[Theorem VI.2.3]{FK2}. Let $\lL_e$ be the symmetric cone of the Peirce 2-space $V_2^e$ for some maximal tripotent $e\in S_r$ \cite{FK2}. Then 
\be{3}\I_V dz\ f(z)=\x{const.}\ \I_{\lL_e}dx\ N_e(x)^b\I_K dk\ f(k\F x)\ee
by \cite[Proposition X.3.4]{FK2} (for the tube domain case $b=0$) and \cite[(2.1.1)]{AU} (for the general case). Applying \er{2} to the right hand side of \er{3} we obtain 
$$\I_V dz\ f(z)=\x{const.}\ \I_{\Rl_{++}^r}\P_{i=1}^r dt_i\ t_i^b\ \P_{1\le i<j\le r}(t_i-t_j)^a\ f^\nl(\F t).$$
\end{proof}

\begin{proposition} For $\l\le r,$ consider the map 
$$\la:\Rl_{++}^\l\to\Rl_{++}^r,\quad\la(t_1,\;,t_\l):=(t_1,\;,t_\l,0^{r-\l}).$$
Then the {\bf Riemann measure} $\ll_\l$ on the Kepler variety $\r V_\l,$ induced by the inner product $(z|w),$ has the radial part 
$\t\ll_\l=\la_*\h M_\l,$ where
\be{5}\h M_\l(dt_1,\;,dt_\l):=\x{const.}\ \P_{i=1}^\l\ dt_i\ t_i^{d_1^c/\l}\ \P_{1\le i<j\le\l}(t_i-t_j)^a\ee
and $d_1^c/\l=b+a(r-\l).$ If $\l=r,$ then $d_1^e=rb$ and \er{5} reduces to \er{4}.
\end{proposition}
\begin{proof} By \cite[Theorem 3.4]{EU} we have
$$\I_{\r V_\l}\ll_\l(dz)\ f(z)=\x{const.}\ \I_{\lL_c^2}dx\ N_c(x)^{d_1^c/\l}\ f^\nl(\F x)
=\x{const.}\ \I_{\Rl_{++}^\l}\P_{i=1}^\l\ dt_i\ t_i^{d_1^c/\l}\ \P_{1\le i<j\le\l}(t_i-t_j)^a\ f^\nl(\F t)$$
by applying \er{2} to the Peirce 2-space $V_2^c$ and its positive cone $\lL_c.$ 
\end{proof}

Let $\PL(V)$ denote the polynomial algebra of a hermitian Jordan triple $V,$ endowed with the Fischer-Fock inner product $(p|q)_V$ for the normalized $K$-invariant inner product $(z|w)$ on $V.$ Let
$$\PL(V)=\S_\m\PL_\m(V)$$
be the {\bf Peter-Weyl decomposition} of $\PL(V)$ under the group $K$ \cite[Theorem 2.1]{FK1}. Here $\m$ runs over the set $\Nl_+^r$ of all integer {\bf partitions} 
$$\m=(m_1\ge\;\ge m_r)$$ 
of length $\le r.$ For a complex parameter $\ln$ let
$$(\ln)_\m=\P_{j=1}^r(\ln-\f a2(j-1))_{m_j}$$ 
denote the multivariate {\bf Pochhammer symbol}. Then the identity
\be{18}(\ln)_{\m+n}=(\ln+n)_\m\ (\ln)_n\ee
holds for any integer $n\ge 0.$

Let $x_1,\;x_h,y_0,\;y_h$ be positive parameters. We say that a $K$-invariant measure $\lm$ supported on $\o\lO$ (or a $K$-invariant subset) is {\bf hypergeometric of type} $\binom{y_0,\;,\ y_h}{x_1,\;,\ x_h}$ if 
\be{6}(p|q)_\lm:=\I{}\lm(dz)\ \o{p(z)}\ q(z)=\f{\P_{i=1}^h(x_i)_\m}{\P_{i=0}^h(y_i)_\m}\ (p|q)_V\ee
for all $\m\in\Nl_+^r$ and $p,q\in\PL_\m(V).$ More generally, for $\l\le r,$ a $K$-invariant measure $\lm$ supported on $\o\lO_\l$ (or a $K$-invariant subset) is {\bf $\l$-hypergeometric} if \er{6} holds for all partitions $\m\in\Nl_+^\l$ of length $\le\l.$ By the Stone-Weierstrass approximation theorem and $K$-invariance, the condition \er{6} determines the measure $\lm$ uniquely, but not every choice of parameters defines such a measure (a kind of multi-variate moment problem).

Let $\lD(z,w)$ be the Jordan triple determinant \cite{FK1}.

\begin{proposition}\label{x} Let $p:=2+a(r-1)+b$ be the genus of $\lO,$ and let $\ln>p-1.$ Then the probability measure $M_\ln$ on $\lO,$ defined by
\be{17}\I_\lO M_\ln(dz)\ f(z)=\x{const.}\I_\lO d\lz\ \lD(\lz,\lz)^{\ln-p}\ f(\lz)\ee
is hypergeometric of type $\binom{\ln}{}.$ 
\end{proposition}
\begin{proof} This follows from the Faraut-Kor\'anyi binomial formula \er{7} proved in \cite{FK1}.
\end{proof}

\begin{proposition}\label{y} For $1\le k\le r$ let $p_k:=2+a(r-k-1)+b$ be the genus for rank $r-k,$ and put
\be{15}\ln_k:=\f dr+\f a2(r-k)=p-1-\f a2(k-1)=1+b+\f a2(2r-k-1)=p_k+\f a2(k+1)-1.\ee
Then the probability measure $M_{k,r}$ on the $k$-th boundary orbit $\lO_{k,r},$ defined in terms of the fibration \er{9} by
\be{8}\I_{\lO_{k,r}}M_{k,r}(dz)\ f(z)=\x{const.}\I_{S_k}dc\I_{\lO^c}d\lz\ \lD(\lz,\lz)^{\ln_k-p_k}\ f(c+\lz)\ee
is hypergeometric of type $\binom{\ln_k}{}.$ 
\end{proposition}
\begin{proof} For the special case $a=2,$ corresponding to the matrix Jordan triple $V=\Cl^{r\xx s},$ this is proved in \cite{BM1} using combinatorial properties of Schur polynomials. The general case \cite[Theorems 6.7 and 6.8]{AU} uses transformation properties under certain non-unimodular groups acting on the boundary.
\end{proof}

For the Shilov boundary $k=r$ $M_{r,r}(dz)$ is the unique $K$-invariant probability measure on $\lO_{r,r}=S,$ since $c+\lO^c=\{c\}$ is a singleton for each $c\in S=S_r.$ For $k=0$ we have $\lO_{0,r}=\lO$ and $p_0=p.$ In this case \er{8} reduces to \er{17} for 
$\ln_0=p-1+\f a2.$ However, in this case we may take any parameter $\ln>p-1.$ Given a frame of minimal orthogonal tripotents 
$e_1,\;,e_r$ of $V$ put
$$c_k:=e_1+\;+e_k.$$
Define
$$\Il_+^r:=\{s\in\Rl^r:\ 1\ge s_1\ge\;\ge s_r\ge 0\}.$$ 
The explicit realization \er{8} of $M_{k,r}$ implies the following proposition: 
\begin{proposition} For $1\le k\le r$ consider the map 
$$\lb:\Il_+^{r-k}\to\Il_+^r,\quad\lb(t_{k+1},\;,t_r):=(1^k,t_{k+1},\;,t_r).$$
Then the $K$-invariant measure $M_{k,r}$ on $\lO_k$ has the radial part $\t M_{k,r}=\lb_*\t M_{\ln_k}^{c_k},$ where 
\be{10}\t M_{\ln_k}^{c_k}(dt_{k+1},\;,dt_r)=\x{const.}\ \P_{i=k+1}^r t_i^b(1-t_i)^{\ln_k-p_k}\ dt_i\ \P_{k<i<j\le r}(t_i-t_j)^a\ee
is the radial part, relative to the Peirce 0-space $V^{c_k}$ of rank $r-k,$ of the weighted Bergman measure $M_{\ln_k}^{c_k}$ for parameter $\ln_k.$ Thus
$$\I_{\lO_{k,r}}M_{k,r}(dz)\ f(z)=\I_{\Il_+^r}\t M_{k,r}(dt_1,\;,dt_r)\ f^\nl(\F{t_1},\;,\F{t_r})
=\I_{\Il_+^r}(\lb_*\t M_{\ln_k}^{c_k})(dt_1,\;,dt_r)\ f^\nl(\F{t_1},\;,\F{t_r})$$
$$=\I_{\Il_+^{r-k}}\t M_{\ln_k}^{c_k}(dt_{k+1},\;,dt_r)\ f^\nl(1^k,\F{t_{k+1}},\;,\F{t_r})$$
$$=\x{const.}\ \I_{\Il_+^{r-k}}\P_{i=k+1}^r t_i^b\ (1-t_i)^{\ln_k-p_k}\ dt_i\ \P_{k<i<j\le r}(t_i-t_j)^a\ f^\nl(1^k,\F{t_{k+1}},\;,\F{t_r}).$$
\end{proposition}

Now let $\l\le r.$ For $\ln>p-1$ define the probability measure
\be{11}M_{\ln,\l}(dz):=\x{const.}\ \lD(z,z)^{\ln-p}\ \ll_\l(dz)\ee
on the Kepler ball $\lO_\l.$ For $\l=r$ we have $\lO_r=\lO$ and recover the ``full" measure $M_{\ln,r}=M_\ln.$ Finally, combining boundary orbits and Kepler varieties, we define the probability measure
\be{12}\I_{\lO_{k,\l}}M_{k,\l}(dz)\ f(z)=\I_{S_k}dc\I_{\lO_{\l-k}^c}M_{\ln_k,\l-k}^c(d\lz)\ f(c+\lz)
=\x{const.}\I_{S_k}dc\I_{\lO_{\l-k}^c}\ll_{\l-k}^c(d\lz)\ \lD(\lz,\lz)^{\ln_k-p_k}\ f(c+\lz)\ee
on $\lO_{k,\l},$ written in terms of the fibration \er{9}. Here $\ll_{\l-k}^c$ is the Riemann measure on the 'little' Kepler ball 
$\lO_{\l-k}^c=\lO^c\ui V_{\l-k}$ induced by the hermitian metric $(z|w)$ restricted to $V^c.$ 

Consider the commuting diagram
$$\xymatrix{\Il_+^{\l-k}\ar[r]^{\lb'}\ar[d]_{\la'}\ar[dr]^{\lg}&\Il_+^\l\ar[d]^{\la}\\\Il_+^{r-k}\ar[r]_{\lb}&\Il_+^r}$$
where
$$\la'(t_{k+1},\;,t_\l):=(1^k,t_{k+1},\;,t_\l)$$
$$\lb'(t_{k+1},\;,t_\l):=(t_{k+1},\;,t_\l,0^{r-\l})$$
$$\lg(t_{k+1},\;,t_\l)=(1^k,t_{k+1},\;,t_\l,0^{r-\l}).$$

\begin{proposition} The $K$-invariant measure $M_{k,\l}$ on $\lO_{k,\l}$ has the radial part 
$\t M_{k,\l}=\lg_*\h M_{k,\l},$ for the measure
\be{13}\h M_{k,\l}(dt_{k+1},\;,dt_\l):=\x{const.}\P_{i=1}^\l\ dt_i\ (1-t_i)^{\ln_k-p_k}\ t_i^{d_1^c/\l}\ \P_{1\le i<j\le\l}(t_i-t_j)^a\ee
on $\Il_+^{\l-k}.$ Thus
$$\I_{\lO_{k,\l}}M_{k,\l}(dz)\ f(z)=\I_{\Il_+^r}\t M_{k,\l}(dt)\ f^\nl(\F t)=\I_{\Il_+^r}(\lg_*\h M_{k,\l})(dt)\ f^\nl(\F t)$$
$$=\I_{\Il_+^{\l-k}}\h M_{k,\l}(dt_{k+1},\;,dt_\l)\ f^\nl(1^k,\F{t_{k+1}},\;,\F{t_\l},0^{r-\l})$$
$$=\x{const.}\I_{\Il_+^{\l-k}}\P_{i=k+1}^\l\ dt_i\ (1-t_i)^{\ln_k-p_k}\ t_i^{d_1^c/\l}\ \P_{k<i<j\le\l}(t_i-t_j)^a
\ f^\nl(1^k,\F{t_{k+1}},\;,\F{t_\l},0^{r-\l})$$
\end{proposition}

Consider the Fischer-Fock kernel $E^\m(z,w)=E_w^\m(z)$ of $\PL_\m(V).$ Then
$$(E_z^\m|E_w^\m)_V=E^\m(z,w).$$
Define $d_\m=\dim\PL_\m(V).$

\begin{lemma} For all $t\in\lD^r$ and $w\in V$ we have
$$(|E_w^\m|^2)^\nl(\F t)=\f{E^\m(w,w)}{d_\m}\  E_e^\m(t).$$ 
\end{lemma}
\begin{proof} Schur orthogonality implies
$$(|E_w^\m|^2)^\nl(\F t)=\I_K dk\ |E^\m(k\F t,w)|^2=\I_K dk\ |(E_{k\F t}^\m|E_w^\m)_V|^2
=\I_K dk\ |(k\.E_{\F t}^\m|E_w^\m)_V|^2=\f{\|E_w^\m\|_V^2\ \|E_{\F t}^\m\|_V^2}{d_\m}$$
Since $\|E_w^\m\|_V^2=E^\m(w,w)$ and $\|E_{\F t}^\m\|_V^2=E^\m(\F t,\F t)=E^\m(t,e),$ the assertion follows.
\end{proof}

\begin{proposition} 
\be{16}\I_{\lD^{r-k}}\P_{i=k+1}^r t_i^b(1-t_i)^{\ln_k-p_k}\ dt_i\ \P_{k<i<j\le r}(t_i-t_j)^a\ E_e^\m(1^k,t_{k+1},\;,t_r)
=\f{d_\m}{(\ln_k)_\m}.\ee
\end{proposition}
\begin{proof} From \er{10} it follows that
$$\I_{\Il_+^{r-k}}\P_{i=k+1}^r t_i^b(1-t_i)^{\ln_k-p_k}\ dt_i\ \P_{k<i<j\le r}(t_i-t_j)^a\ E_e^\m(1^k,t_{k+1},\;,t_r)
=\I_{\Il_+^r}\t M_{k,r}(dt)\ E^\m(t,e)$$
$$=\f{d_\m}{E^\m(e,e)}\I_{\Il_+^r}\t M_{k,r}(dt)\ (|E_e^\m|^2)^\nl(\F t)
=\f{d_\m}{E^\m(e,e)}\I_{\lO_k}M_{k,r}(dz)\ |E_e^\m(z)|^2=\f{d_\m}{\|E_e^\m\|_V^2}\ \|E_e^\m\|_{\ln_k}^2=\f{d_\m}{(\ln_k)_\m}.$$
\end{proof}

\begin{remark} In the special case $V=\Cl^{r\xx s}$ the polynomials $E_e^\m$ are proportional to the Schur polynomials, and the identity \er{16} was shown directly in \cite{BM1}. A direct proof of \er{16} in the general case would be of interest.
\end{remark}

The following theorem is our first main result.

\begin{theorem}\label{t} For $1\le k\le\l\le r$ the probability measure $M_{k,\l}$ on $\lO_{k,\l}$ is $\l$-hypergeometric of type
$\binom{\f dr,\ r\f a2,\ \ln_k}{\l\f a2,\ \ln_\l}.$
\end{theorem}
\begin{proof} Let $c=c_\l.$ Put $h:=d_1^c/\l=b+a(r-\l).$ Applying \er{16} to the Jordan triple $V_2^c$ (of tube type) we obtain for 
$\m\in\Nl_+^\l,$ putting $d_\m^c=\dim\PL_\m(V_2^c),$
$$\x{const.}\I_{\lD^{\l-k}}\P_{i=k+1}^\l(1-t_i)^{\ln_k-p_k}\ dt_i\ \P_{k<i<j\le\l}(t_i-t_j)^a\ E_{c_\l}^\m(1^k,t_{k+1},\;,t_\l)
=\f{d_\m^c}{(1+\f a2(2\l-k-1))_\m}=\f{d_\m^c}{(\ln_k-h)_\m}$$
since
$$1+\f a2(2\l-k-1)+h=1+\f a2(2\l-k-1)+b+a(r-\l)=1+b+\f a2(2\r-k-1)=\ln_k$$
For $z\in V_2^c$ we have $E_c^\m(z)=E^\m(c,c)\ \lF_\m^c(z),$ where $\lF_\m^c\in\PL_\m(V_2^c)$ is the spherical polynomial normalized by 
$\lF_\m^c(c)=1.$ Therefore
$$N_c(z)^h\ E_c^\m(z)=E^\m(c,c)\ N_c(z)^h\ \lF_\m^c(z)=E^\m(c,c)\ \lF_{\m+h}^c(z)=\f{E^\m(c,c)}{E^{\m+h}(c,c)}\ E_c^{\m+h}(z).$$
We have
$$E^\m(c,c)=\f{d_\m^c}{(1+\f a2(\l-1))_\m}$$
and, similarly,
$$E^{\m+h}(c,c)=\f{d_{\m+h}^c}{(1+\f a2(\l-1))_{\m+h}}=\f{d_\m^c}{(\ln_\l-h)_{\m+h}},$$
since
$$1+\f a2(\l-1)+h=1+\f a2(\l-1)+b+a(r-\l)=1+b+\f a2(2\r-l-1)=\ln_\l.$$
It follows that
$$\P_{i=k+1}^\l t_i^h\ (|E_c^\m|^2)^\nl(1^k,\F{t_{k+1}},\;,\F{t_\l},0^{r-\l})
=\f{E^\m(c,c)}{d_\m}N_c(1^k,t_{k+1},\;,t_\l)^h\ E_c^\m(1^k,t_{k+1},\;,t_\l)$$
$$=\f{E^\m(c,c)}{d_\m}\ \f{E^\m(c,c)}{E^{\m+h}(c,c)}\ E_c^{\m+h}(1^k,t_{k+1},\;,t_\l)
=\f{E^\m(c,c)}{d_\m}\ \f{(\ln_\l-h)_{\m+h}}{(\ln_\l-h)_\m}\ E_c^{\m+h}(1^k,t_{k+1},\;,t_\l).$$
Applying \er{16} to $\m+h\in\Nl_+^\l$ we obtain
$$\f1{\x{const.}}\|E_c^\m\|_{\ln_k,\l}^2=\f1{\x{const.}}\I_{\lO_{k,\l}}M_{k,\l}(dz)\ |E_c^\m(z)|^2$$
$$=\I_{\Il_+^{\l-k}}\P_{i=k+1}^\l t_i^h(1-t_i)^{\ln_k-p_k}\ dt_i\ \P_{k<i<j\le\l}(t_i-t_j)^a
\ (|E_c^\m|^2)^\nl(1^k,\F{t_{k+1}},\;,\F{t_\l},0^{r-\l})$$
$$=\f{E^\m(c,c)}{d_\m}\ \f{(\ln_\l-h)_{\m+h}}{(\ln_\l-h)_\m}
\I_{\lD^{\l-k}}\P_{i=k+1}^\l(1-t_i)^{\ln_k-p_k}\ dt_i\ \P_{k<i<j\le\l}(t_i-t_j)^a\ E_c^{\m+h}(1^k,t_{k+1},\;,t_\l)$$
$$=\f{E^\m(c,c)}{d_\m}\ \f{(\ln_\l-h)_{\m+h}}{(\ln_\l-h)_\m}\ \f{d_{\m+h}^c}{(\ln_k-h)_{\m+h}}
=\f{E^\m(c,c)}{(\ln_k-h)_{\m+h}}\ \f{(\ln_\l-h)_{\m+h}}{(\ln_\l-h)_\m}\ \f{(a\l/2)_\m}{(ar/2)_\m}\ \f{(\ln_\l-h)_\m}{(d/r)_\m}$$
using the identity
$$\f{d_{\m+h}^c}{d_\m}=\f{d_\m^c}{d_\m}
=\f{(a\l/2)_\m}{(ar/2)_\m}\ \f{(1+\f a2(\l-1))_\m}{(d/r)_\m}=\f{(a\l/2)_\m}{(ar/2)_\m}\ \f{(\ln_\l-h)_\m}{(d/r)_\m}$$
as computed in the proof of \cite[Theorem 5.1]{EU}. Simplifying and using \er{18} we finally obtain
$$\|E_c^\m\|_{k,\l}^2=E^\m(c,c)\ \f{(\ln_\l)_\m}{(\ln_k)_\m\ (d/r)_\m}\ \f{(a\l/2)_\m}{(ar/2)_\m}$$
since $M_{k,\l}$ is a probability measure. It follows that for $\m\in\Nl_+^\l$ and $p,q\in\PL_\m(V)$ we have
$$(p|q)_{k,\l}:=\I_{\lO_{k,\l}}M_{k,\l}(dz)\ \o{p(z)}\ q(z)=(p|q)_V\ \f{(\ln_\l)_\m}{(\ln_k)_\m\ (d/r)_\m}\ \f{(a\l/2)_\m}{(ar/2)_\m}.$$
\end{proof}

\section{Holomorphic Function Spaces and Toeplitz Operators}
We now define Hilbert spaces of holomorphic functions and Toeplitz type operators associated with hypergeometric measures of rank 
$\l\le r,$ keeping in mind the examples $M_{k,\l}$ on $\o\lO_{k,\l}$ constructed above. For $\l\le r$ define
$$\PL^\l(V)=\S_{\m\in\Nl_+^\l}\PL_\m(V),$$
involving only partitions of length $\le\l.$ Then the restriction map $p\mapsto p|_{V_\l}$ is injective and yields a linear isomorphism between $\PL^\l(V)$ and the regular functions on the Kepler variety $V_\l.$ For a $K$-invariant $\l$-hypergeometric measure $\lm$ on 
$\o\lO_{k,\l}$ let $\HL_{\lm,\l}$ denote the Hilbert space of all holomorphic functions on the Kepler ball $\lO_\l$ which are square-integrable under the measure $\lm.$ This is the completion of $\PL^\l(V),$ restricted to $\lO_\l,$ for the measure $\lm.$

This general definition covers all classical examples. Consider first the ``full" case $\l=r.$ For a discrete series Wallach parameter $\ln>p-1,$ the {\bf weighted Bergman space} $\HL_\ln$ consists of all holomorphic functions on $\lO$ which are square-integrable under the measure $M_\ln.$ For $1\le k\le r$ the {\bf embedded Wallach parameters} $\ln_k$ defined in \er{15} belong to the continuous Wallach set 
\be{14}\ln>\f a2(r-1)\ee 
but not to the discrete series since $k\ge 1$ implies $\ln_k\le 1+b+\f a2(2r-2)=p-1.$ The associated {\bf Hardy type spaces} 
$\HL_{k,r}$ consist of all holomorphic functions on $\lO$ which are square-integrable under the measure $M_{k,r}.$ Then $\ln_r=\f dr$ is the ``true" Hardy space parameter, corresponding to the Shilov boundary $S=\lO_{r,r}.$ The left endpoint $\ln_1=p-1$ of the holomorphic discrete series corresponds to the probability measure $M_{1,r}$ on the dense open boundary orbit $\lO_{1,r}.$ As explained in Section \ref{r}, the parameters $\ln_k$ are of special importance for subnormal $G$-homogeneous Toeplitz operators. By Proposition \ref{x} and Proposition \ref{y}, these measures are of hypergeometric type. 

Now consider the ``partial" case $\l\le r.$ If $\ln>p-1,$ the {\bf partial weighted Bergman space} $\HL_{\ln,\l}$ consists of all holomorphic functions on the Kepler ball $\lO_\l$ which are square-integrable for the probability measure $M_{\ln,\l}.$ The inner product is
$$(\lf|\lq)_{\ln,\l}:=\I_{\lO_\l}M_{\ln,\l}(dz)\ \o{\lf(z)}\ \lq(z)
=\x{const.}\I_{\lO_\l}\ll_\l(dz)\ \lD(z,z)^{\ln-p}\ \o{\lf(z)}\ \lq(z).$$
For $\l=r$ we have $\lO_r=\lO$ and $M_{\ln,r}=M_\ln.$ Thus we recover the 'full' weighted Bergman space $\HL_{\ln,r}=\HL_\ln.$  For 
$1\le k\le\l\le r,$ the {\bf partial Hardy type space} $\HL_{k,\l}$ consists of all holomorphic functions on the Kepler ball $\lO_\l$ which are square-integrable for the probability measure $M_{k,\l}.$ The inner product is
$$(\lf|\lq)_{k,\l}:=\I_{\lO_{k,\l}}M_{k,\l}(dz)\ \o{\lf(z)}\ \lq(z)
=\I_{S_k}dc\I_{\lO_{\l-k}^c}\ll_{\l-k}^c(d\lz)\ \lD(\lz,\lz)^{\ln_k-p_k}\ (\o\lf\lq)(c+\lz).$$
Putting $\l=r$ we recover the inner product \er{16} since $\lO_{r-k}^c=\lO^c$ and $M_{r-k}^c(d\lz)=d\lz$ is the Lebesgue measure on 
$V^c.$ For $k=0$ we have $c=0,\ V^0=V,\ \lO_\l^0=\lO_\l=\lO\ui V_\l,\ M_\l^0=M_\l$ and $p_0=p.$ Thus we recover the $M_\l$-inner product. 

In summary, we obtain examples of type $\binom{\f dr,\ r\f a2,\ \ln_k}{\l\f a2,\ \ln_\l}$ for $0\le k\le\l\le r.$ For fixed $\l$ we have as special cases the partial weighted Bergman spaces of type $\binom{\f dr,\ r\f a2,\ \ln}{\l\f a2,\ \ln_\l},$ corresponding to 
$k=0,$ and the partial Hardy space of type $\binom{\f dr,\ r\f a2}{\ln_\l}$ corresponding to maximal $k=\l.$ For $\l=r$ we obtain the full type $\binom{\ln_k}{},$ since $\ln_r=\f dr,$ specializing to the full weighted Bergman spaces of type $\binom{\ln}{}$ if 
$k=0$ and the full Hardy space of type $\binom{\f dr}{}$ if $k=r.$ It would be interesting to construct natural examples of more complicated hypergeometric type.  

We now introduce Toeplitz operators in our setting. For the 'full' Hilbert space $\HL_\lm$ over $\lO$ we denote by 
$P_\lm:L^2(\o\lO,\lm)\to\HL_\lm$ the orthogonal projection and define the ``full" Toeplitz operator $T_\lm(f),$ with symbol function 
$f\in L^\oo(\o\lO),$ by
$$T_\lm(f)=P_\lm\ f\ P_\lm.$$
Restricting to continuous symbols we obtain the ``full" Toeplitz $C^*$-algebra
$$\TL_\lm=C^*(T_\lm(f):\ f\in\CL(\o\lO)).$$
As special cases, we obtain the ``full" Bergman-Toeplitz operators $T_{\ln,r}(f)$ ($\ln>p-1$) and the ``full" Hardy type Toeplitz operators $T_{k,r}(f)$ ($1\le k\le r$) associated with the hypergeometric measures $M_{\ln,r}$ on $\lO$ and $M_{k,r}$ on $\lO_{k,r},$ respectively. The corresponding Toeplitz $C^*$-algebras are denoted by $\TL_{\ln,r}$ and $\TL_{k,r},$ respectively.  

In the more general setting of the ``partial" Hilbert space $\HL_{\lm,\l}$ over $\lO_\l,$ associated with a $K$-invariant $\l$-hypergeometric measure $\lm$ ($\l\le r$), denote by $P_{\lm,\l}:L^2(\o\lO_\l,\lm)\to\HL_{\lm,\l}$ the orthogonal projection and define the ``partial" Toeplitz operator $T_{\lm,\l}(f),$ with symbol function $f\in L^\oo(\o\lO_\l),$ by
$$T_{\lm,\l}(f)=P_{\lm,\l}\ f\ P_{\lm,\l}.$$
Restricting to continuous symbols we obtain the ``partial" Toeplitz $C^*$-algebra
$$\TL_{\lm,\l}=C^*(T_{\lm,\l}(f):\ f\in\CL(\o\lO_\l)).$$
As special cases, we obtain the ``partial" Bergman-Toeplitz operators 
$T_{\ln,\l}(f)$ ($\ln>p-1$) and the ``partial" Hardy type Toeplitz operators $T_{k,\l}(f)$ ($1\le k\le\l$) associated with the $\l$-hypergeometric measures $M_{\ln,\l}$ on $\lO_\l$ and $M_{k,\l}$ on $\lO_{k,\l},$ respectively. The corresponding Toeplitz $C^*$-algebras are denoted by $\TL_{\ln,\l}$ and $\TL_{k,\l},$ respectively.  

\begin{lemma} Let $p,q\in\PL(V).$ Then the Toeplitz type operators satisfy
$$T_{\lm,\l}(p)\ T_{\lm,\l}(q)=T_{\lm,\l}(pq).$$
\end{lemma}
\begin{proof} Since $\PL^\l(V)^\perp$ is an ideal in $\PL(V)$ it follows that
$$T_{\lm,\l}(pq)\lf=P_{\lm,\l}(pq\lf)=P_{\lm,\l}(p(P_{\lm,\l}+P_{\lm,\l}^\perp)(q\lf))$$
$$=P_{\lm,\l}(p\ P_{\lm,\l}(q\lf))+P_{\lm,\l}(p\ P_{\lm,\l}^\perp(q\lf))=P_{\lm,\l}(p\ T_{\lm,\l}(q)\lf)
=T_{\lm,\l}(p)(T_{\lm,\l}(q)\lf).$$
\end{proof}

It follows that $\TL_{\lm,\l}$ is generated by Toeplitz type operators with linear symbols and their adjoints. 

\begin{remark}\label{f} A standard reproducing kernel argument (carried out in \cite[Proposition 4.2]{U5}) shows, at least for the 'concrete' hypergeometric measures described above (where the support is connected), that the $C^*$-algebra $\TL_{\lm,\l}$ acts {\bf irreducibly} on $\HL_{\lm,\l}.$
\end{remark}

For any $v\in V$ let
$$v^*(z):=(z|v)$$
denote the associated linear form. Its conjugate is $\o v^*(z)=\o{(z|v)}=(v|z).$ Let $\dl_v p(z):=p'(z)v$ denote the directional derivative. Put
$$\Le_j:=(0,\;,0,1,0,\;,0)$$
with $1$ at the $j$-th place. It is shown in \cite[Corollary 2.10]{U2} that
\be{22}v^*p\in\S_{j=1}^r\PL_{\m+\Le_j}(V),\quad\dl_v p\in\S_{j=1}^r\PL_{\m-\Le_j}(V)\ee
for all $p\in\PL_\m(V),$ with zero-component if $\m\pm\Le_j$ is not a partition. Let $q\mapsto q_\m\in\PL_\m(V)$ denote the $\m$-th isotypic projection.

The next result determines the fine structure of the adjoint Toeplitz type operator $T_{\lm,\l}(v^*)^*=T_{\lm,\l}(\o v^*).$ 

\begin{proposition}\label{e} Let $\lm$ be a $\l$-hypergeometric measure on $\o\lO_\l.$ Let $v\in V.$ Then
$$T_{\lm,\l}(\o v^*)p=\S_{j=1}^\l\ \f{\P_{i=1}^h(x_i-\f a2(j-1)+m_j-1)}{\P_{i=0}^h(y_i-\f a2(j-1)+m_j-1)}\ (\dl_v p)_{\m-\Le_j}$$
for all $\m\in\Nl_+^\l$ and $p\in\PL_\m(V).$
\end{proposition}
\begin{proof} Let $q\in\PL_\n(V),\ \n\in\Nl_+^\l,$ satisfy $(T_{\lm,\l}(\o v^*)p|q)_{\lm,\l}\ne0.$ Then
$$(p|v^*q)_{\lm,\l}=(T_{\lm,\l}(\o v^*)p|q)_{\lm,\l}\ne0.$$
With \er{22} it follows that $\m=\n+\Le_j$ for some $j\le\l$ and hence $\n=\m-\Le_j.$ Since $\lm$ is $\l$-hypergeometric, it follows that
\be{25}(T_{\lm,\l}(\o v^*)p|q)_\lm=(p|v^*q)_\lm=\f{\P_{i=1}^h(x_i)_\m}{\P_{i=0}^h(y_i)_\m}\ (p|v^*q)_V
=\f{\P_{i=1}^h(x_i)_\m}{\P_{i=0}^h(y_i)_\m}\ (\dl_v p|q)_V
=\f{\P_{i=1}^h(x_i)_\m/(x_i)_{\m-\Le_j}}{\P_{i=0}^h(y_i)_\m/(y_i)_{\m-\Le_j}}\ (\dl_v p|q)_\lm.\ee
Since $q$ is arbitrary, it follows that
$$T_\lm^\l(\o v^*)p=\S_{j=1}^\l\f{\P_{i=1}^h(x_i)_\m/(x_i)_{\m-\Le_j}}{\P_{i=0}^h(y_i)_\m/(y_i)_{\m-\Le_j}}\ (\dl_v p)_{\m-\Le_j}.$$
Now the assertion follows from
$$\f{(\ll)_\m}{(\ll)_{\m-\Le_j}}=\f{(\ll-\f a2(j-1))_{m_j}}{(\ll-\f a2(j-1))_{m_j-1}}=\ll-\f a2(j-1)+m_j-1.$$
\end{proof}

\section{\bf Limit measures}
The basic result concerning Toeplitz $C^*$-algebras on bounded symmetric domains states that every irreducible representation is realized on a unique boundary component $\lO^c,$ for any tripotent $c.$ This was carried out in full detail for the Hardy space in \cite{U2,U3} and its generalization to weighted Bergman spaces was described in \cite{U4}. Here a crucial step, which was indicated in \cite{U4} and proved in detail in the recent paper \cite{U5}, is the limit behavior of the underlying measures under certain 
{\bf peaking functions}. In the present paper, this crucial result will be generalized to the boundary orbits $\lO_{k,\l},$ and their intersection with Kepler varieties. This is not completely straightforward, since the assignment $f^{(c)}(\lz):=f(c+\lz)$ is not compatible with the Peter-Weyl decomposition of $\PL(V).$ 

Let $c\in S_i$ with $i\le\l.$ Since $V_2^c=P_2^c V$ has rank $i\le\l$ and $(z|c)^n=(P_2^c z|c)^n,$ where $P_2^c$ denotes the Peirce 2-projection, it follows that
$$(z|c)^n\in\PL(V_2^c)\ic\PL^i(V)\ic\PL^\l(V).$$
Restricting (injectively) to $\lO_\l,$ the holomorphic function
\be{28}H_c(z):=\exp(z|c)=\S_{n=0}^\oo\f{(z|c)^n}{n!}\ee
on $\lO_\l$ can be regarded as an element of the Hilbert completion $\HL_{\lm,\l}$ of $\PL^\l(V)$ under $\lm.$ This applies in particular to $i=1.$  

Let $0\le i\le\l\le r$ and $c\in S_i.$ Then $c+\o\lO^c\ic\o\lO.$ For functions $f\in\CL(\o\lO_\l)$ we define 
$f^{(c)}\in\CL(\o\lO_{\l-i}^c)$ by 
\be{42}f^{(c)}(\lz):=f(c+\lz)\quad(\lz\in\o\lO_{\l-i}^c).\ee

\begin{lemma}\label{f} Let $\lm$ be an $\l$-hypergeometric measure on $\o\lO_\l.$ Let $0\le i\le\l$ and $c\in S_i.$ Then
$$\lim_{n\to\oo}\I_{\o\lO_\l}\lm(dz)\ \f{|H_c^n(z)|^2}{\|H_c^n\|_\lm^2}\ f(z)=0$$
for all $f\in\CL(\o\lO_\l)$ satisfying $f^{(c)}=0.$
\end{lemma}
\begin{proof} By assumption, for every $\Le>0$ there is an open neighborhood $U\ic\o\lO_\l$ of $c+\lO_{\l-i}^c$ satisfying 
$\sup |f(U)|\le\Le.$ By \cite[Lemma 6.2]{Lo2} we have $|(z|c)|<(c|c)$ for all $z\in\o\lO\sm\o\lO^c.$ Peirce orthogonality implies 
$(z|c)=(c|c)$ for all $z\in c+\o\lO^c.$ Therefore $|H_c|<H_c(c)$ on $\o\lO_\l\sm U,$ and a compactness argument shows that there exists an open neighborhood $V\ic U\ic\o\lO_\l$ of $c+\o\lO_{\l-i}^c$ such that
$$q:=\f{\sup_{\o\lO_\l\sm U}|H_c|}{\inf_V|H_c|}<1.$$
Therefore
$$\I_{\o\lO_\l}\lm(dz)\ \f{|H_c^n(z)|^2}{\|H_c^n\|_\lm^2}\ f(z)=\I_U\lm(dz)\ \f{|H_c^n(z)|^2}{\|H_c^n\|_\lm^2}\ f(z)
+\I_{\o\lO_\l\sm U}\lm(dz)\ \f{|H_c^n(z)|^2}{\|H_c^n\|_\lm^2}\ f(z)$$
$$\le\sup_U|f|+\sup_{\o\lO_\l}|f|\.\f{\I_{\o\lO_\l\sm U}\lm(dz)\ |H_c^n(z)|^2}{\I_V\lm(dz)\ |H_c^n(z)|^2}
\le\Le+\sup_{\o\lO_\l}|f|\.q^{2n}\ \f{\x{Vol}_\lm(\o\lO_\l\sm U)}{\x{Vol}_\lm(V)}.$$  
Since $q^{2n}\to 0$ it follows that
$$\limsup_{n\to\oo}\I_{\o\lO_\l}\lm(dz)\ \f{|H_c^n(z)|^2}{\|H_c^n\|_\lm^2}\ f(z)\le\Le.$$
\end{proof}

Now consider the special case $i=1.$ For $c=e_1\in S_1,$ let $\la:=(\la_1,\ldots,\la_{\l-1})\in\Nl_+^{\l-1}$ be a partition of length 
$\l-1.$ Define
\be{28}\la^+:=(\la_1,\la)\in\Nl_+^\l\ee
and consider the conical function
$$N_{\la^+}=N_2^{\la_1-\la_2}\ N_3^{\la_2-\la_3}\: N_\l^{\la_{\l-1}},$$
where $N_1,\;,N_r$ are the Jordan theoretic minors \cite{U1}. Then the conical function $N_\la^c$ relative to $V^c$ for the partition $\la$ satisfies
$$N_{\la^+}^{(c)}=N_\la^c.$$

The asymptotic expansion of generalized hypergeometric series 
\be{23}\z{p}{F}{q}(z)=\S_{n=0}^\oo\f{\P_{r=1}^p\lG(n+\lb_r)}{\P_{r=1}^q\lG(n+\lm_r)}\ \f{z^n}{n!}\ee
in one variable $z$ has been determined in \cite{W}. Put $\lk:=1+q-p$ and
$$\lt:=\f{q-p}2+\lb_1+\;+\lb_p-\lm_1-\;-\lm_q.$$
As a special case $M=1$ of \cite[Theorem 1]{W}, using \cite[Lemma 1]{W}, one obtains
$$\lim_{x\to+\oo}X^{-\lt}\ e^{-X}\ \z{p}{F}{q}(x)=A_0=(2\lp)^{(p-q)/2}\ \lk^{\f12-\lt},$$
where $X:=\lk\ x^{1/\lk}.$ If $q=p+1,$ this simplifies to $\lk=2,\ X=2\F x$ and $A_0=(2\lp)^{-1/2}\ 2^{\f12-\lt}=\lp^{-1/2}\ 2^{-\lt}.$ Therefore
\be{24}\lim_{x\to\oo}x^{-\lt/2}\ e^{-2\F x}\ \z{p}{F}{q}(x)=\f1{\F\lp}.\ee

\begin{theorem}\label{t} Let $\lm$ be a $K$-invariant $\l$-hypergeometric probability measure of type $\binom{y_0,\;,y_h}{x_1;\;,x_h}$ on $\o\lO_\l.$ Then for each $c\in S_1$ there exists a unique $K^c$-invariant $(\l-1)$-hypergeometric probability measure $\lm^{(c)}$ of type $\binom{y_0-\f a2,\;,y_h-\f a2}{x_1-\f a2,\;,x_h-\f a2}$ on $\o\lO_{\l-1}^c$ such that for all continuous functions $f$ we have
\be{27}\lim_{n\to\oo}\I_{\o\lO_\l}\lm(dz)\ \f{|H_c^n(z)|^2}{\|H_c^n\|_\lm^2}\ f(z)=\I_{\o\lO_{\l-1}^c}\lm^{(c)}(d\lz)\ \ f^{(c)}(\lz).\ee
\end{theorem}
\begin{proof} By $K$-invariance, we may assume that $c=e_1.$ By Lemma \ref{f} each weak cluster point $\lm'$ of the sequence of probability measures on the left of \er{27} is supported on the closure $\o\lO_{\l-1}^c$ and is invariant under $K^c.$ Thus it suffices to compute the $\lm'$-inner product for $\la$-homogeneous polynomials on $V^c,$ where $\la\in\Nl_+^{\l-1}$ is arbitrary. By irreducibility, it is enough to consider the conical functions $N_\la^c$ relative to $V^c.$ Defining $\la^+\in\Nl_+^\l$ as in 
\er{28}, we consider for any $s\in\Nl$ the conical function 
$$(z|e_1)^s\ N_{\la^+}=N_1^s\ N_{\la^+}=N_\m,$$ 
where $\m=(m_1,\la_1,\;,\la_{\l-1},0^{r-\l})$ and $m_1=s+\la_1.$ In the proof of \cite[Theorem 5.5]{U5} it was shown that the respective Fock inner products are related by
$$\f{\|N_\m\|_V^2}{\|N_\la^c\|_{V^c}^2}=\f{(1+\f a2(\l-1))_\m}{(1+\f a2(\l-2))_\la}
\P_{1\le j<\l}\f{(1+\f a2(j-1))_{m_1-\la_j}}{(1+\f a2 j)_{m_1-\la_j}}
=(1+\f a2(\l-1))_{m_1}\P_{1\le j<\l}\f{(1+\f a2(j-1))_{m_1-\la_j}}{(1+\f a2 j)_{m_1-\la_j}}.$$
For any $\ll\in\Cl$ we have
$$\f{(\ll)_\m}{(\ll-\f a2)_\la}=(\ll)_{m_1}\P_{1<j\le\l}\f{(\ll-\f a2(j-1))_{m_j}}{(\ll-\f a2-\f a2(j-2))_{\la_{j-1}}}=(\ll)_{m_1}.$$
It follows that
$$\f{\|N_\m\|_\lm^2}{\|N_\la^c\|_{V^c}^2}=\f{\|N_\m\|_V^2}{\|N_\la^c\|_{V^c}^2}\ \f{\P_{i=1}^h(x_i)_\m}{\P_{i=0}^h(y_i)_\m}
=\f{\P_{i=1}^h(x_i)_\m}{\P_{i=0}^h(y_i)_\m}\ (1+\f a2(\l-1))_{m_1}\P_{1\le j<\l}\f{(1+\f a2(j-1))_{m_1-\la_j}}{(1+\f a2 j)_{m_1-\la_j}}$$
$$=\f{\P_{i=1}^h(x_i-\f a2)_\la}{\P_{i=0}^h(y_i-\f a2)_\la}
\ \f{(1+\f a2(\l-1))_{m_1}\P_{i=1}^h(x_i)_{m_1}}{\P_{i=0}^h(y_i)_{m_1}}\P_{1\le j<\l}\f{(1+\f a2(j-1))_{m_1-\la_j}}{(1+\f a2 j)_{m_1-\la_j}}=A\ \f{\P_{i=1}^h(x_i-\f a2)_\la}{\P_{i=0}^h(y_i-\f a2)_\la}\ B(m_1),$$
where $A$ is independent of $\la$ and $s,$ and 
$$B(t):=\f{\lG(t+1+\f a2(\l-1))\P_{i=1}^h\lG(t+x_i)}{\P_{i=0}^h\lG(t+y_i)}
\P_{1\le j<\l}\f{\lG(t+1+\f a2(j-1)-\la_j)}{\lG(t+1+\f a2 j-\la_j)}.$$
For $(e^{(z|e_1)})^n=e^{n(z|e_1)}$ we obtain by orthogonality
$$\f1{\|N_\la^c\|_{V^c}^2}\I_{\o\lO^\l}\lm(dz)\ |e^{(z|e_1)}|^{2n}\ |N_{\la^+}(z)|^2
=\S_{s\ge 0}\f{n^{2s}}{(s!)^2}\f1{\|N_\la^c\|_{V^c}^2}\I_{\o\lO^\l}\lm(dz)\ |(z|e_1)|^{2s}\ |N_{\la^+}(z)|^2$$
$$=\S_{s\ge 0}\f{n^{2s}}{(s!)^2}\f{\|N_\m\|_\lm^2}{\|N_\la^c\|_{V^c}^2}
=A\ \f{\P_{i=1}^h(x_i-\f a2)_\la}{\P_{i=0}^h(y_i-\f a2)_\la}\S_{s\ge 0}\f{n^{2s}}{(s!)^2}B(\la_1+s)
=A\ \f{\P_{i=1}^h(x_i-\f a2)_\la}{\P_{i=0}^h(y_i-\f a2)_\la}\ F_\la(n^2),$$
where $F_\la(X)$ is a hypergeometric series in the sense of \er{23}, with parameters 
$$\la_1+x_1,\;,\ \la_1+x_h,\ \la_1+1+\f a2(\l-1),\ \la_1-\la_2+1+\f a2,\;,\ \la_1-\la_{\l-1}+1+\f a2(\l-2)$$ 
in the numerator and 
$$\la_1+y_0,\;\ \la_1+y_h,\ 1+\f a2,\ \la_1-\la_2+1+\f a2\ 2,\;,\ \la_1-\la_{\l-1}+1+\f a2(\l-1)$$ 
in the denominator. One power of $s!$ cancels against the numerator term $\lG(1+\f a2(j-2)+\la_1-\la_{j-1}+s)$ for $j=2.$ The crucial parameter $\lt$ in \er{24} is computed as
$$\lt=\f12+\S_{i=1}^h(\la_1+x_i)+\(\la_1+1+\f a2(\l-1)\)+\(\la_1-\la_2+1+\f a2\)+\;+\(\la_1-\la_{\l-1}+1+\f a2(\l-2)\)$$
$$-\S_{i=0}^h(\la_1+y_i)-\(1+\f a2\)-\(\la_1-\la_2+1+\f a2\ 2\)-\;-\(\la_1-\la_{\l-1}+1+\f a2(\l-1)\)$$
$$=\f12+\S_{i=1}^h x_i-\S_{i=0}^h y_i+\(1+\f a2(\l-1)\)-\(1+\f a2\)-\f a2(\l-2)=\f12+\S_{i=1}^h x_i-\S_{i=0}^h y_i.$$
Putting $x=n^2,$ \er{24} implies
$$\lim_{n\to\oo}n^{-\lt}\ e^{-2n}\ F_\la(n^2)=\f1{\F\lp}.$$
Since $\lt$ is independent of $\la,$ the same limit holds for $\la=0.$ Thus we obtain
$$\lim_{n\to\oo}\f{F_\la(n^2)}{F_0(n^2)}=1.$$
Passing to the probability measure cancels the constant $A$ and we obtain
$$\f1{\|N_\la^c\|_{V^c}^2}\I_{\lO^\l}\lm(dz)\ \f{|e^{(z|e_1)}|^{2n}}{\|(e^{(z|e_1)})^n\|_\lm^2}\ |N_{\la^+}(z)|^2
\to\f{\P_{i=1}^h(x_i-\f a2)_\la}{\P_{i=0}^h(y_i-\f a2)_\la}.$$
Hence any cluster point $\lm'$ is an $(\l-1)$-hypergeometric probability measure of the same type 
$\binom{y_0-\f a2,\;,y_h-\f a2}{x_1-\f a2,\;,x_h-\f a2}$ on $\o\lO_{\l-1}^c.$ In view of Lemma \ref{f} this determines the limit measure on each irreducible $K^c$-type, which, as explained above, implies the assertion.
\end{proof}

\begin{remark} For the ``concrete" $\l$-hypergeometric measures $M_{\ln,\l}$ ($k=0$) and $M_{k,\l}$ ($k>0$) constructed in Section 
\ref{s} we obtain as limit measures
$$M_{\ln,\l}^{(c)}=M_{\ln-\f a2,\l-1}^c$$ 
$$M_{k,\l}^{(c)}=M_{k-1,\l-1}^c,$$
where the superscript $c$ refers to the Peirce 0-space $V^c.$ In the second case this follows from 
$$\ln_k-\f a2=\ln_{k-1}^c.$$
If $k=0$ then $\ln>p-1$ is any parameter in the discrete series, in which case $\ln-\f a2>p^{(c)}-1$ belongs to the discrete series of 
$\lO^c.$ As special cases ($\l=r$) we have
$$M_\ln^{(c)}=M_{\ln-\f a2}^c$$ 
$$M_{k,r}^{(c)}=M_{k-1,r-1}^c$$
for the ``full" measures. Here for $k\ge 2$ and $\x{rank}\ \lO^c=r-1$ the value
$$\ln_{k-1}^c=1+b+\f a2(2(r-1)-(k-1)-1)=1+b+\f a2(2r-k-1)-\f a2=\ln_k^r-\f a2$$
is again a boundary parameter for $\lO^c,$ whereas for $k=1$ the parameter
$$\ln_0=\ln_1-\f a2=p-1-\f a2=1+b+a(r-1)-\f a2>1+b+a(r-2)=p_{r-1}-1$$
belongs to the discrete series of $\lO^c.$ Understanding this ``disappearing boundary orbit" in the limit was one of the original motivations for the current paper.
\end{remark}

\section{\bf Boundary representations}
The (unital) Toeplitz $C^*$-algebra $\TL$ associated with a bounded domain $\lO\ic\Cl^d$ can be regarded as a deformation of 
$\CL(\o\lO)$ in the sense of ``non-commutative geometry". Thus the spectrum of $\TL,$ consisting of all irreducible $*$-representations, is a 'non-commutative' (non-Hausdorff) compactification of $\lO,$ involving the geometry of the boundary. In this section we carry out this program for Toeplitz operators over boundary orbits and algebraic varieties, using the boundary stratification described in Proposition \ref{d}. For each $0\le j<k$ the partial closures satisfy
$$\lO_{k,r}=\U_{c\in S_j}c+\lO_{k-j,r-j}^c.$$
as a non-disjoint union.

For two sequences $(f_n),(g_n)$ in $\HL_{\lm,\l}$ we put
$$f_n\sim g_n$$
if $\lim_{n\to\oo}\|f_n-g_n\|_{\lm,\l}=0.$ For any $c\in S_i$ put 
$$h_c^n(z):=H_c^n(z)/\|H_c^n\|_{\lm,\l}.$$ 
In the following we embed $\PL(V^c)\ic\PL(V)$ via the Peirce projection $V\to V^c.$

\begin{lemma}\label{h} Let $p\in\PL^\l(V)$ and $q\in\PL^{\l-1}(V^c)\ic\PL^\l(V).$ Then
$$T_{\lm,\l}(p)(h_c^n\ q)\sim h_c^n\ T_{\lm^c,\l-1}(p^{(c)})q$$
for all $c\in S_1$
\end{lemma}
\begin{proof} Since $p-p^{(c)}$ vanishes on $c+\lO_{\l-1}^c,$ Lemma \ref{f} implies 
$$\|\f{H_c^n}{\|H_c^n\|_{\lm,\l}}\ p-\f{H_c^n}{\|H_c^n\|_{\lm,\l}}\ p^{(c)}\|_{\lm,\l}^2
=\I_{\o\lO_\l}\lm(dz)\ \f{|H_c^n(z)|^2}{\|H_c^n\|_{\lm,\l}^2}\ |p(z)-p^{(c)}(z)|^2\to 0.$$
It follows that
$$T_{\lm,\l}(p)(h_c^n\ q)=p(h_c^n\ q)\sim h_s^n(p^{(c)}\ q)\sim h_c^n\ T_{\lm^c,\l-1}^c(p^{(c)})q.$$
\end{proof}

The adjoint operators $T_{\lm,\l}(\o p)$ are more difficult to handle. For a partition $\la=(\la_1,\;,\la_{\l-1})\in\Nl_+^{\l-1}$ consider the orthogonal projection
$$\lp_\la^\l:\PL^\l(V)\to\S_{m_1\ge\la_1}\PL_{m_1,\la}(V)\ic\PL^\l(V),$$
with $(m_1,\la)\in\Nl_+^\l\ic\Nl_+^r.$ Then $\S_{\la\in\Nl_+^{\l-1}}\lp_\la^\l=\x{Id}$ on $\PL^\l(V).$

\begin{lemma}\label{n} Let $p\in\PL(V_2^c)\ic\PL^\l(V)$ and $v\in V^c,$ where $c=e_1.$ Then we have for every $\la\in\Nl_+^{\l-1}$
\be{39}p\ N_{\la^+}\in\x{Ran}(\lp_\la^\l)\ee
\be{40}T_{\lm,\l}(\o v^*)(p\ N_{\la^+})=\S_{j=1}^{\l-1}\f{\P_{i=1}^h(x_i-\f a2-\f a2(j-1)+\la_j-1)}
{\P_{i=0}^h(y_i-\f a2-\f a2(j-1)+\la_j-1)}\ \lp_{\la-\Le_j}^\l(p\.\dl_v N_{\la^+}).\ee
\end{lemma}
\begin{proof} The first assertion is proved in \cite[Lemma 3.5]{U2}. By \cite[Lemma 2.9]{U2} we have
$$\dl_v N_\m\in\S_{j=2}^\l\PL_{\m-\Le_j}(V).$$
Since $v\in V^c$ implies $\dl_v p=0,$ we have $\dl_v N_\m=p\.\dl_v N_{\la^+},$ and Proposition \ref{e} yields
$$T_{\lm,\l}(\o v^*)(p\ N_{\la^+})=T_{\lm,\l}(\o v^*)\ N_\m
=\S_{j=2}^\l\f{\P_{i=1}^h(x_i-\f a2(j-1)+m_j-1)}{\P_{i=0}^h(y_i-\f a2(j-1)+m_j-1)}\ (\dl_v N_\m)_{\m-\Le_j}$$
$$=\S_{j=2}^\l\f{\P_{i=1}^h(x_i-\f a2(j-1)+\la_j^+-1)}{\P_{i=0}^h(y_i-\f a2(j-1)+\la_j^+-1)}\ (p\.\dl_v N_{\la^+})_{\m-\Le_j}.$$
Shifting $j\mapsto j-1$ and using $\PL_{\m-\Le_j}(V)\ic\text{Ran}(\lp_{\la-\Le_{j-1}}^\l)$ for all $1<j\le\l,$ the assertion follows. 
\end{proof}

\begin{lemma}\label{i} Let $q\in\PL^{\l-1}(V^c)$ and $\la\in\Nl_+^{\l-1}.$ Then
$$\lp_\la^\l(h_c^n\ q)\sim h_c^n\ q_\la.$$
\end{lemma}
\begin{proof} We may assume that $q\in\PL_\lb(V^c)$ for some partition $\lb\in\Nl_+^{\l-1}.$ Every $\lg\in K^c$ has an extension 
$g\in K$ satisfying $gc=c$ (see the proof of \cite[Lemma 6.2]{U5}). Since $h_c^n$ is fixed under the action of $g,$ we may assume that $q=N_\lb'$ is the conical polynomial in $V^c$ of type $\lb.$ Then $N_{\lb^+}-q$ vanishes on $c+\lO_{\l-1}^c,$ and Lemma \ref{f} implies 
\be{41}h_c^n\ q\sim h_c^n\ N_{\lb^+}.\ee
Since the projection $\lp_\la^\l$ has a continuous extension to $\HL_{\lm,\l}$ it follows that
$$\lp_\la^\l(h_c^n\ q)\sim\lp_\la^\l(h_c^n\ N_{\lb^+}).$$
Since $h_c^n$ belongs to the closure of $\PL(V_2^c)$ in $\HL_{\lm,\l},$ \er{39} implies $h_c^n\ N_{\lb^+}\in\x{Ran}(\lp_\lb^\l).$ Therefore orthogonality implies
$$\lp_\la^\l(h_c^n\ N_{\lb^+})=\ld_{\la,\lb}\ h_c^n\ N_{\lb^+}\sim\ld_{\la,\lb}\ h_c^n\ q=h_c^n\ q_\la.$$
\end{proof}

\begin{proposition}\label{k} Let $p\in\PL^\l(V)$ and $q\in\PL^{\l-1}(V^c)\ic\PL^\l(V).$ Then the adjoint Toeplitz operators satisfy
$$T_{\lm,\l}(\o p)(h_c^n\ q)\sim h_c^n\ T_{\lm^{(c)},\l-1}^c(\o p^{(c)})q$$
for all $c\in S_1.$
\end{proposition}
\begin{proof} Assume first that $p(z)=(z|v)$ is linear. If $v\in V_2^c\op V_1^c,$ then $p^{(c)}$ is constant and Lemma \ref{h} implies
$$T_{\lm,\l}(\o p)(h_c^n\ q)=P_{\lm,\l}(\o p\ h_c^n\ q)\sim P_{\lm,\l}(\o p^{(c)}\ h_c^n\ q)
=\o p^{(c)}\ h_c^n\ q=h_c^n\ T_{\lm^{(c)},\l-1}^c(\o p^{(c)})q$$
since the orthogonal projection $P_{\lm,\l}$ is continuous. If $v\in V^c,$ we may assume as in the proof of Lemma \ref{i} that 
$q=N_\la'$ is the conical polynomial in $\PL_\la(V^c)$ for some partition $\la\in\Nl_+^{\l-1}.$ Then $N_{\la^+}-q$ vanishes on 
$c+\lO_{\l-1}^c.$ Since $v$ is tangent to $V^c$ it follows that $(\dl_v N_{\la^+})^c=\dl_v N_\la^c.$ Hence 
$\dl_v(N_{\la^+}-q)$ vanishes on $c+\lO_{\l-1}^c$ as well. Applying \er{41}, Lemma \ref{n} and Lemma \ref{i}, we obtain 
$$T_{\lm,\l}(\o v^*)(h_c^n\ q)\sim T_{\lm,\l}(\o v^*)(h_c^n\ N_{\la^+})
=\S_{j=2}^\l\ \f{\P_{i=1}^h(x_i-\f a2-\f a2(j-1)+\la_j-1)}{\P_{i=0}^h(y_i-\f a2-\f a2(j-1)+\la_j-1)}
\ \lp_{\la-\Le_{j-1}}^\l(h_c^n\.\dl_v N_{\la^+})$$
$$\sim\S_{j=2}^\l\ \f{\P_{i=1}^h(x_i-\f a2-\f a2(j-1)+\la_j-1)}{\P_{i=0}^h(y_i-\f a2-\f a2(j-1)+\la_j-1)}
\ \lp_{\la-\Le_{j-1}}^\l(h_c^n\.\dl_v q)$$
$$\sim h_c^n\S_{j=2}^\l\ \f{\P_{i=1}^h(x_i-\f a2-\f a2(j-1)+\la_j-1)}{\P_{i=0}^h(y_i-\f a2-\f a2(j-1)+\la_j-1)}
\ (\dl_v q)_{\la-\Le_{j-1}}=h_c^n\ T_{\lm^{(c)},\l-1}^c(\o p^{(c)})(q),$$
since $r-j=(r-1)-(j-1)$ and $\l-j=(\l-1)-(j-1).$ The last identity follows from Proposition \ref{e} and the fact that $p^{(c)}=p$ if 
$v\in V^c.$ This proves the assertion for linear symbol functions. 

Now suppose that the assertion holds for polynomials $\lf,\lq$ up to a certain degree. Since $\lm^c$ is again a $(\l-1)$-hypergeometric measure for $V^c$ and $\lf^{(c)}$ has degree $\le\deg\lf$, we may apply this assumption to $q$ and 
$T_{\lm^c,\l-1}^c(\o\lf^{(c)})q\in\PL^{\l-1}(V^c)$ to obtain
$$T_{\lm,\l}(\o{\lf\lq})(h_c^n\ q)=T_{\lm,\l}(\o\lq)\ T_{\lm,\l}(\o\lf)(h_c^n\ q)\sim T_{\lm,\l}(\o\lq)(h_c^n\ T_{\lm^{(c)},\l-1}^c(\o\lf^{(c)})q)$$
$$\sim h_c^n\ T_{\lm^c,\l-1}^c(\o\lq^s)\ T_{\lm^{(c)},\l-1}^c(\o\lf^{(c)})q=h_c^n\ T_{\lm^{(c)},\l-1}^c(\o{\lf\lq}^c)q.$$
Thus the assertion holds for $\lf\lq.$ Since the assertion holds for linear forms, the proof is complete.
\end{proof}

The following is our main result.

\begin{theorem}\label{l} Let $0\le i\le\l$ and let $c\in S_i$ be arbitrary. Then the Toeplitz $C^*$-algebra $\TL_{k,\l}$ has an irreducible $*$-representation
$$\ls_{k,\l}^{(c)}:\TL_{k,\l}\to\TL_{k\sm i,\l-i}^c$$
which is uniquely determined by the property 
\be{45}\ls_{k,\l}^{(c)}T_{k,\l}(f)=T_{k\sm i,\l-i}^c(f^{(c)})\ee
for all $f\in\CL(\lO_{k,\l}),$ with $f^{(c)}\in\CL(\lO_{k\sm i,\l-i}^c)$ defined by \er{42}. Here we define
$$k\sm i:=\begin{cases}k-i&i<k\\0&k\le i\le\l\end{cases}.$$ 
In the first case the Toeplitz operator $T_{k-i,\l-i}^c$ acts on a boundary orbit of the ``little" Kepler ball $\lO_{\l-i}^c.$ In the second case the Toeplitz operator $T_{0,\l-i}^c=T_{\l-i}^c$ acts on $\lO_{\l-i}^c=\lO_{0,\l-i}^c$ with discrete series parameter $\ln_k-i\f a2.$ 
\end{theorem}
\begin{proof} For orthogonal tripotents $c\in S_i,d\in S_j^c,$ the defining property \er{45} yields a commuting diagram
$$\xymatrix{&\TL_{k\sm i,\l-i}^c\ar[dd]^{(\ls^c)_{k\sm i,\l-i}^{(d)}}\\\TL_{k,\l}\ar[ur]^{\ls_{k,\l}^{(c)}}\ar[dr]_{\ls_{k,\l}^{(c+d)}}&
\\&\TL_{k\sm{(i+j)},\l-(i+j)}^{c+d}}.$$
Since every tripotent is the orthogonal sum of minimal tripotents, it therefore suffices to consider minimal tripotents $c\in S_1.$ We may also assume $k\ge 1,$ since the Kepler ball case $k=0$ has been proven in \cite{U5}. 

Let $\AL$ denote the set of all operators $A$ in the $*$-subalgebra $\TL_0\ic\TL_{k,\l}$ generated by polynomial symbols, such that there exists an operator $A_c$ acting on $\PL(V^c)$ which satisfies
\be{37}\lim_{n\to\oo}\|A(h_c^n\ q)-h_c^n\ (A_c q)\|_{k,\l}=0\ee
for all $q\in\PL(V^c)\ic\PL^\l(V).$ Theorem \ref{t} implies that $A_c$ is uniquely determined by $A$ and 
\be{36}\|A_c\|\le\|A\|\ee
for the respective operator norms. By definition, $\AL$ is an algebra and \er{36} implies that $A\mapsto A_c$ has an extension
$\AL\to\BL(\HL_{k-1,\l-1}^c)$ (bounded operators) which is an algebra homomorphism. For every $p\in\PL(V),$ it follows from Lemma 
\ref{h} that $T_{k,\l}(p)\in\AL$ and $(T_{k,\l}p)_c=T_{k-1,\l-1}^c p^{(c)}.$ The corresponding statement $T_{k,\l}(\o p)\in\AL$ and 
$(T_{k,\l}\o p)_c=T_{k-1,\l-1}^c\o p^{(c)}$ for the adjoint operator follows from the deeper Proposition \ref{k}. Thus we have 
$\AL=\TL_0$ and, by \er{36}, $A\mapsto A_c$ has a unique $C^*$-extension, denoted by $\ls_{k,\l}^{(c)}$ to the closure $\TL_{k,\l}$ of 
$\TL_0.$ This extension satisfies \er{37} for all continuous symbols $f,$ since this property holds for polynomials and their conjugates. Thus we obtain a $C^*$-homomorphism
$$\ls_{k,\l}^{(c)}:\TL_{k,\l}\to\TL_{k-1,\l-1}^c.$$
As mentioned above, the case for arbitrary tripotents follows by iteration. The irreducibility of these representations follows from Remark \ref{f} applied to $M_{k\sm i,\l-i}^c.$ 
\end{proof}

\begin{remark} For different tripotents $c\in S_i$ and $d\in S_j$ the representations $\ls^{(c)}$ and $\ls^{(d)}$ are inequivalent.
This follows from Urysohn's Lemma since there exists $f\in\CL(\lO_{k,\l})$ which vanishes on 
$c+\lO_{k\sm i,\l-i}^c$ but not on $d+\lO_{k\sm j,\l-j}^d.$ Hence $T_{k,\l}(f)$ belongs to $\text{Ker}(\ls^{(c)})$ but not to 
$\text{Ker}(\ls^{(d)}).$ With more effort one can show that the full spectrum of $\TL_{k,\l}$ is given by the representations constructed above.
\end{remark}


\begin{thebibliography}{}
\bibitem{AU} J. Arazy, H. Upmeier, \emph{Boundary measures for symmetric domains and integral formulas for the discrete Wallach points}, Int. Eq. Op. Th. \textbf{47} (2003), 375-434
\bibitem{AZ} J. Arazy and G. Zhang, {\emph Homogeneous multiplication operators on bounded symmetric domains},  J. Functional Anal. \textbf{202} (2003), 44-66.
\bibitem{BM1} B. Bagchi, G. Misra, \emph{Homogeneous tuples of multiplication operators on twisted Bergman spaces}, J. Functional Anal. \textbf{136} (1996), 171-213
\bibitem{BM2} \bysame, \emph{Homogeneous tuples of operators and systems of imprimitivity, Contemp. Math.}, \textbf{185} (1995), 67-76.
\bibitem{B} N. Bourbaki, \emph{\'El\'ements de Math\'ematique, Livre VI, Int\'egration}, Chapitres 7 \& 8, Herman (1963).
\bibitem{BD} J. Bunce, J. Deddens, \emph{On the normal spectrum of a subnormal operator}, Proc. Amer. Math. Soc. \textbf{63} (1977), 107-110.
\bibitem{CY} S. Chavan, D. Yakubovich, \emph{Spherical Tuples of Hilbert Space Operators}, Indiana Univ. Math. J. \textbf{64} (2015), 577-612.
\bibitem{EU} M. Englis, H. Upmeier, \emph{\bf Reproducing kernel functions and asymptotic expansions on Jordan-Kepler varieties}, Advances in Math. \textbf{347} (2019), 780-826
\bibitem{FK1} J. Faraut, A. Kor\'anyi, \emph{Function spaces and reproducing kernels on bounded symmetric domains}, J. Funct. Anal. \textbf{88} (1990), 64-89.
\bibitem{FK2} J. Faraut, A. Kor\'anyi, \emph{Analysis on Symmetric Cones}, Clarendon Press, Oxford (1994).
\bibitem{GKP} S. Ghara, S. Kumar, and P. Pramanick, $\Kl$ \emph{-homogeneous tuple of operators on bounded symmetric domains}, work in progress. 
\bibitem{I} T. Inoue, \emph{Unitary representations and kernel functions associated with the boundaries of a bounded symmetric domain}, Hiroshima Math. J. \textbf{10} (1980), 75-140.
\bibitem{Ka} K. Kadell, \emph{The Selberg-Jack symmetric functions}, Adv. in Math. \textbf{130} (1997), 33-102 
\bibitem{KM} A. Koranyi, G. Misra, \emph{Homogeneous Hermitian holomorphic vector bundles and the Cowen-Douglas class over bounded symmetric domains}, Adv. in Math., \textbf{351} (2019), 1105-1138
\bibitem{Kn} A. Knapp, \emph{Representation Theory of Semisimple Groups: An Overview Based on Examples}, Princeton Univ. Press, 1986.
\bibitem{La1} M. Lassalle, \emph{Noyau de Szeg\"o, $K$-types et alg\`ebres de Jordan}, C. R. Acad. Sci. Paris \textbf{303} (1986), 1-4.
\bibitem{La2} M. Lassalle, \emph{Alg\`ebres de Jordan et ensemble de Wallach}, Invent. Math. \textbf{89} (1987), 375-393.
\bibitem{Lo1} O. Loos, \emph{Jordan Pairs}, Springer Lect. Notes in Math. \textbf{460} (1975).
\bibitem{Lo2} O. Loos, \emph{Bounded Symmetric Domains and Jordan Pairs}, Univ. of California, Irvine (1977).
\bibitem{L} A.R. Lubin, \emph{Spectral inclusion and C.N.E.}, Can. J. Math. \textbf{34} (1982),  883-887.
\bibitem{M} G. Misra, \emph{Curvature and the backward shift operator}, Proc. Amer. Math. Soc. \textbf{91} (1984), 105-107.
\bibitem{MS} G. Misra, N.S.N. Sastry, \emph{Homogeneous tuples of operators and holomorphic discrete series representation of some classical groups}, J. Operator Th. \textbf{24} (1990), 23-32.
\bibitem{MU} G. Misra, H. Upmeier, \emph{Singular Hilbert modules on Jordan-Kepler varieties}, Proc. IWOTA (Shanghai 2018), to appear
\bibitem{RV} H. Rossi, M. Vergne, \emph{Analytic continuation of the holomorphic Discrete series of a semi-simple Lie group}, Acta. Math. \textbf{136} (1976), 1-59. 
\bibitem{Sch} W. Schmid, \emph{Die Randwerte holomorpher Funktionen auf hermitesch symmetrischen R\"aumen}, Invent. Math. 
\textbf{8} (1969), 61-80.
\bibitem{St} R. P. Stanley, \emph{Some combinatorial properties of the Jack symmetric functions}, Adv. Math. \textbf{77} (1989), 76-115.
\bibitem{U1} H. Upmeier, \emph{Jordan algebras and harmonic analysis on symmetric spaces}, Amer. J. Math. \textbf{108} (1986), 1-25.
\bibitem{U2} H. Upmeier, \emph{Toeplitz operators on bounded symmetric domains}, Trans. Amer. Math. Soc. \textbf{280} (1983), 221-237.
\bibitem{U3} H. Upmeier, \emph{Toeplitz $C^*$-algebras on bounded symmetric domains}, Annals of Math. \textbf{119} (1984), 549-576
\bibitem{U4} H. Upmeier, \emph{Toeplitz Operators and Index Theory in Several Complex Variables}, Operator Theory: Advances and Applications, Vol. 81, Birkh\"auser Verlag, Basel-Boston-Berlin, 1996
\bibitem{U5} H. Upmeier, \emph{Toeplitz operators on Jordan-Kepler varieties}. Banach J. of Math. Anal. \textbf{13} (2019), 314-337
\bibitem{V} V.S. Varadarajan, \emph{Geometry of Quantum Theory}, Springer Verlag, New York, 1985.
\bibitem{W} N. Wallach, \emph{The analytic continuation of the discrete series, I, II}, Trans. Amer. Math. Soc. \textbf{251} (1979), 1-17 and 19-37.
\bibitem{W} E.M. Wright, \emph{The asymptotic expansion of the generalized hypergeometric function}. Proc. London Math. Soc. 
\textbf{46} (1940), 389-408
\end{thebibliography}
\end{document}